\documentclass[draft]{amsart}

\usepackage{amsmath}
\usepackage{amsthm}
\usepackage{amssymb}
\usepackage{url}

\usepackage{enumerate}
\usepackage[mathscr]{eucal}
\usepackage{mathrsfs}
\usepackage{verbatim}

\theoremstyle{plain}
\newtheorem{theorem}{Theorem}[section]
\newtheorem{lemma}[theorem]{Lemma}
\newtheorem{proposition}[theorem]{Proposition}
\newtheorem{corollary}[theorem]{Corollary}
\newtheorem{claim}[theorem]{Claim}

\theoremstyle{remark}
\newtheorem{remark}[theorem]{Remark}

\numberwithin{equation}{section}
\numberwithin{theorem}{section}

\newcommand{\supp}{\operatorname{supp}}
\newcommand{\dist}{\operatorname{dist}}

\makeatletter
\@namedef{subjclassname@2010}{%
  \textup{2010} Mathematics Subject Classification}
\makeatother

\begin{document}

\subjclass[2010]{Primary 11P21, 42B10, 42B20. Secondary 11L07}
\keywords{Lattice points, convex domains, Fourier transform, Van der
Corput's method.}

\date{}

\title[Lattice Points]{Lattice points in rotated convex domains}

\author[J.W. Guo]{Jingwei Guo}

\address{Jingwei Guo\\Department of Mathematics\\
University of Illinois at Urbana-Champaign\\
Urbana, IL 61801, USA}

\email{jwguo@illinois.edu}

\thanks{}

\begin{abstract}
If $\mathcal{B}\subset \mathbb{R}^d$ ($d\geqslant 2$) is a compact
convex domain with a smooth boundary of finite type, we prove that
for almost every rotation $\theta\in SO(d)$ the remainder of the
lattice point problem, $P_{\theta \mathcal{B}}(t)$, is of order
$O_{\theta}(t^{d-2+2/(d+1)-\zeta_d})$ with a positive number
$\zeta_d$. Furthermore we extend the estimate of the above type, in
the planar case, to general compact convex domains.
\end{abstract}

\maketitle


\section{Introduction}\label{introduction}

Let $\mathcal{B}\subset \mathbb{R}^d$ ($d\geqslant 2$) be a compact
convex domain, which contains the origin in its interior and has a
smooth boundary $\partial \mathcal{B}$. The number of lattice points
$\mathbb{Z}^d$ in the dilated domain $t\mathcal{B}$ is approximately
$|t\mathcal{B}|$ ({\it i.e.} the volume (area if $d=2$) of
$t\mathcal{B}$) and the lattice point problem is to study the
remainder, $P_{\mathcal{B}}(t)$, in the equation
\begin{equation*}
P_{\mathcal{B}}(t)=\#(t\mathcal{B}\cap\mathbb{Z}^d)-|\mathcal{B}|t^d
\quad \textrm{for $t\geqslant 1$}.
\end{equation*}
A trivial estimate gives $P_{\mathcal{B}}(t)=O(t^{d-1})$.

If $\partial \mathcal{B}$ has everywhere positive (Gaussian)
curvature, a standard estimate is
\begin{equation*}
P_{\mathcal{B}}(t)=O(t^{d-2+2/(d+1)}),
\end{equation*}
which can be readily obtained by a combination of the Poisson
summation formula and (nowadays standard) oscillatory integral
estimates (see Hlawka~\cite{Hlawka}). Over the years this result has
been improved by many authors and the best bounds up-to-date are due
to Huxley~\cite{huxley2003} in the planar case and Guo~\cite{guo} in
the higher dimensional case. For a survey on historical results the
reader is referred to Ivi\'c, Kr\"atzel, K\"uhleitner, and
Nowak~\cite{survey2004}.

While the above case is relatively well understood, the general case
when the (Gaussian) curvature is allowed to vanish is not.

Let us first consider the $d\geqslant 3$ case of vanishing
curvature. Partial results indicate that the remainder may become
much larger. For example Randol~\cite{randol_II} considered the
super spheres
\begin{equation*}
\mathcal{B}=\{(x_1, \ldots, x_d)\in \mathbb{R}^d :
|x_1|^{\omega}+|x_2|^{\omega}+\cdots+|x_d|^{\omega}\leqslant 1\}
\end{equation*}
for even integer $\omega\geqslant 3$, and proved that
\begin{equation*}
P_{\mathcal{B}}(t)=\bigg\{ \begin{array}{ll}
O(t^{d-2+2/(d+1)}) & \textrm{for $\omega\leqslant d+1$},\\
O(t^{(d-1)(1-1/\omega)}) & \textrm{for $\omega>d+1$},
\end{array}
\end{equation*}
and this estimate is the best possible when $\omega>d+1$.
Kr\"atzel~\cite{kratzel_odd} extended this result to odd
$\omega\geqslant 3$ and gave an asymptotic formula
\begin{equation}
P_{\mathcal{B}}(t)=H(t)t^{(d-1)(1-1/\omega)}+O(t^\Theta)
\label{asymptotics}
\end{equation}
with an explicit $\Theta<(d-1)(1-1/\omega)$ and $H(t)$ continuous
and periodic (see Kr\"atzel~\cite{kratzel} for more details). We
observe that the remainder $P_{\mathcal{B}}(t)$ becomes extremely
large as $\omega\rightarrow \infty$.

This observation is supported by the study of more examples, and
special attention is paid to specific convex domains in
$\mathbb{R}^3$. See Kr\"atzel~\cite{kratzel_2002} and Kr{\"a}tzel
and Nowak~\cite{K-N-2008,K-N-2011}, in which they proved, among
other results, asymptotic formulas of $P_{\mathcal{B}}(t)$ with
explicit representations of the main terms given.

For general domains with boundary points of Gaussian curvature zero,
our knowledge is still very poor. Partial results in $\mathbb{R}^3$
are available in Kr\"atzel~\cite{kratzel_2000, kratzel_2002_MM},
Peter~\cite{peter_2002}, Popov~\cite{popov}, and
Nowak~\cite{nowak_2008} (with the latter two papers focusing on
bodies of rotation). Under a variety of assumptions, they provide
$O$-estimates (or asymptotic formulas) of $P_{\mathcal{B}}(t)$, and
evaluate the contributions (to $P_{\mathcal{B}}(t)$) of different
types of boundary points of Gaussian curvature zero.  Their results
show that the size of $P_{\mathcal{B}}(t)$ depends on certain
properties of the boundary points of Gaussian curvature zero and
whether the slope of the normal at such a point is rational or
irrational. In particular $P_{\mathcal{B}}(t)$ may become extremely
large and a substantial contribution to it is due to the
neighborhoods of those boundary points of Gaussian curvature zero at
which the normal has a rational direction.

However after a rotation of the domain there may be no such points,
hence we can expect a better estimate. For example Iosevich, Sawyer,
and Seeger~\cite{I-S-S-0} proved, for convex domains of finite type
(\footnote{Check \S \ref{geometry} for its definition.}), that there
is $r>2$ so that
\begin{equation}
P_{\mathcal{B}_\theta}(t)=O_{\theta}(t^{d-2+2/(d+1)}\log^{1/r}(2+t))
\quad \textrm{for a.e. $\theta\in SO(d)$}, \label{ISS_result}
\end{equation}
where $\mathcal{B}_{\theta}=\theta \mathcal{B}$ denotes the rotated
domain $\{\theta x : x\in \mathcal{B}\}$. Results of type
\eqref{ISS_result} with the same exponent $d-2+2/(d+1)$ can be found
in Randol~\cite{randol_R^n} for convex domains with an analytic
boundary, and in Colin de Verdi{\`e}re~\cite{colin} for general (not
necessarily convex) domains if $d\leqslant 7$.

 It
is then natural to ask whether one can prove a result of type
\eqref{ISS_result} with an exponent $d-2+2/(d+1)-c$ for some
positive $c$. We made a progress in this direction and proved the
following theorem with a $c>0$ depending only on the dimension $d$.

\begin{theorem}\label{Rd-ae}
Let $\mathcal{B}\subset \mathbb{R}^d$ ($d\geqslant 3$) be a compact
convex domain containing the origin in its interior. If the boundary
is a smooth hypersurface of finite type then
\begin{equation*}
P_{\mathcal{B}_\theta}(t)=O_{\theta}(t^{d-2+2/(d+1)-\zeta_d}) \quad
\textrm{for a.e. $\theta\in SO(d)$},
\end{equation*}
where $\zeta_d>0$ is defined as
\begin{equation}
\zeta_d=\left\{ \begin{array}{ll} \frac{2(d-2)(d-1)d}{(d+1)(6d^5+118d^4+109d^3-210d^2-119d+82)} & \textrm{for $3\leqslant d\leqslant 4$},\\
\frac{(d-3)(d-1)d}{2d^6+49d^5+123d^4-9d^3-167d^2-52d+30} &
\textrm{for $d\geqslant 5$}.
\end{array}\right. \label{zeta_d}
\end{equation}
\end{theorem}

This result is an easy consequence of the following theorem.

\begin{theorem}\label{Rd-finitetype}
Let $\mathcal{B}\subset \mathbb{R}^d$ ($d\geqslant 3$) be a compact
convex domain containing the origin in its interior. If the boundary
is a smooth hypersurface of finite type $\omega$ then
\begin{equation*}
\sup\limits_{t\geqslant 2}
|P_{\mathcal{B}_\theta}(t)|/(t^{d-2+2/(d+1)-\zeta_d-\sigma(d,
\omega)}\log^{b}(t))\in  L^1(SO(d)),
\end{equation*}
where $b>1$, $\zeta_d$ is given by \eqref{zeta_d}, and $\sigma(d,
\omega)>0$ is defined as
\begin{equation}
\sigma(d, \omega)=\left\{ \begin{array}{ll}
\frac{4d(6d^5+100d^4-230d^3-193d^2+496d-172)}{(6d^5+118d^4+109d^3-210d^2-119d+82)\cdot\Box}
& \textrm{for
$3\leqslant d\leqslant 4$},\\
\frac{2d(2d^5+39d^4-105d^3-205d^2+377d-96)}
{(2d^5+47d^4+76d^3-85d^2-82d+30)\cdot\triangle}& \textrm{for
$d\geqslant 5$},
\end{array}\right.\label{sigma_d_w}
\end{equation}
with
\begin{equation}
\begin{split}
\Box&=6(\omega-2)d^5+118(\omega-2)d^4+109(\omega-2)d^3\\
&\quad -6(35\omega-71)d^2+(246-119\omega)d+(82\omega-156)
\end{split}\label{square}
\end{equation}
and
\begin{equation}
\begin{split}
\triangle&=2(\omega-2)d^5+47(\omega-2)d^4+76(\omega-2)d^3\\
&\quad +(172-85\omega)d^2+(166-82\omega)d+(30\omega-56).
\end{split}\label{triangle}
\end{equation}
\end{theorem}

The proof of Theorem \ref{Rd-finitetype} relies on the following
analysis result (implied by Svensson's \cite[Theorem
4.1]{svensson}): if $\mathcal{B}\subset \mathbb{R}^d$ ($d\geqslant
3$) is a compact convex domain and its boundary is a smooth
hypersurface of finite type $\omega$ (\footnote{The restriction on
the size of $\omega$ given by Svensson's \cite[Theorem
4.1]{svensson} can be removed under current assumptions.}) then
\begin{equation}
\Phi\in L^p(S^{d-1})\qquad \textrm{for $2<p< 2+2/(d-1)(\omega-2)$},
\label{svenssonThm}
\end{equation}
where
\begin{equation}
\Phi(\xi)=\sup\limits_{r>0}
r^{(d+1)/2}|\widehat{\chi}_{\mathcal{B}}(r\xi)|, \qquad
\textrm{$\xi\in S^{d-1}$}. \label{def-Phi}
\end{equation}

For a general convex domain $\mathcal{B}$ with a smooth boundary,
\eqref{svenssonThm} is not necessarily true, however, we always have
(due to Varchenko's \cite[Theorem 8]{varchenko}) that
\begin{equation*}
r^{(d+1)/2}|\widehat{\chi}_{\mathcal{B}}(r\xi)|\in L^2(S^{d-1}).
\end{equation*}
By using this result we can readily modify the proof of Theorem
\ref{Rd-finitetype} and prove the following theorem, which improves
similar results contained in Randol~\cite[p.~285]{randol_R^n} and
Varchenko's \cite[Theorem 7]{varchenko} in terms of the estimate.
\begin{theorem}\label{Rd-varchenkoplus}
Let $\mathcal{B}\subset \mathbb{R}^d$ ($d\geqslant 3$) be a compact
convex domain containing the origin in its interior. If the boundary
is a smooth hypersurface then
\begin{equation*}
|P_{\mathcal{B}_\theta}(t)|/t^{d-2+2/(d+1)-\zeta_d}\in L^1(SO(d)),
\end{equation*}
where $\zeta_d$ is given by \eqref{zeta_d}.
\end{theorem}

Let us now consider the $d=2$ case of vanishing curvature, in which
we have a better understanding than in higher dimensions. We refer
the interested readers to Ivi\'c, Kr\"atzel, K\"uhleitner, and
Nowak~\cite{survey2004} and Guo~\cite{guo2} for an introduction to
related results.

For general convex planar domains we know $\Phi\in L^{2,
\infty}(S^1)$ (Brandolini, Colzani, Iosevich, Podkorytov, and
Travaglini's \cite[Theorem 0.3]{brandolini}). By using this result
and the same method used in the proofs of Theorem
\ref{Rd-finitetype} and \ref{Rd-varchenkoplus}, we are able to
extend our previous result for convex planar domains of finite type
in Guo's \cite[Theorem 1.1]{guo2} to the following result for convex
planar domains with no curvature assumption on the boundary (with
even a better estimate, due to an improved estimate of certain
nonvanishing determinants given in Lemma \ref{r2-keylemma} below).

\begin{theorem}\label{R2-general}
If $\mathcal{B}$ is a compact convex planar domain with a smooth
boundary containing the origin in its interior, then
\begin{equation*}
\sup\limits_{t\geqslant 2} |P_{\mathcal{B}_\theta}(t)|/
(t^{2/3-\zeta_2}\log^{b}(t)) \in L^1(SO(2)),
\end{equation*}
where $b>1$ and $\zeta_2=1/2859$. In particular,
\begin{equation*}
P_{\mathcal{B}_\theta}(t)=O_{\theta}(t^{2/3-\zeta_2}\log^{b} (t))
\quad \textrm{for a.e. $\theta\in SO(2)$}.
\end{equation*}
\end{theorem}

This theorem improves Iosevich's \cite[Theorem 0.2]{iosevich} and
also Brandolini, Colzani, Iosevich, Podkorytov, and Travaglini's
\cite[Theorem 0.1]{brandolini} (in terms of the estimate). If we
assume that the boundary is of finite type then we have the
following better estimate (again due to the improved result given in
Lemma \ref{r2-keylemma} below), which improves Guo's \cite[Theorem
1.2]{guo2}.

\begin{theorem}\label{R2-finitetype}
If $\mathcal{B}$ is a compact convex planar domain with a smooth
boundary of finite type $\omega$ containing the origin in its
interior, then
\begin{equation*}
\sup\limits_{t\geqslant 2} |P_{\mathcal{B}_\theta}(t)|/
(t^{2/3-\zeta_2-\sigma(2, \omega)}\log^{b}(t)) \in L^1(SO(2)),
\end{equation*}
where $b>1$, $\zeta_2=1/2859$, and
\begin{equation*}
\sigma(2, \omega)=\frac{616}{953(953\omega-1848)}.
\end{equation*}
In particular,
\begin{equation*}
P_{\mathcal{B}_\theta}(t)=O_{\theta}(t^{2/3-\zeta_2}) \quad
\textrm{for a.e. $\theta\in SO(2)$}.
\end{equation*}
\end{theorem}

\begin{remark}
Our main idea originates from Iosevich, Sawyer, and
Seeger~\cite[p.~168-169]{I-S-S-0} and M\"{u}ller~\cite{mullerII}.
Our main tools used in this paper are from the oscillatory integral
theory and the classical Van der Corput's method of exponential sums
(namely, the $A$- and $B$-processes). To prove our estimate of
exponential sums (see Proposition \ref{exp-sum} below) we use an
$A^{q}B$-process. If we use more $A$- and $B$-processes we may
achieve further improvement at the cost of more technical
difficulties.
\end{remark}

{\it Notations:} We use the usual Euclidean norm $|x|$ for a point
$x\in \mathbb{R}^d$. $B(x, r)$ represents the Euclidean ball
centered at $x$ with radius $r$, and its dimension will be clear
from the context. The norm of a matrix $A\in \mathbb{R}^{d\times d}$
is given by $\|A\|=\sup_{|x|=1}|Ax|$. We set $e(f(x))=\exp(-2\pi i
f(x))$, $\mathbb{Z}_{*}^{d}=\mathbb{Z}^{d}\setminus \{0\}$, and
$\mathbb{R}^d_*=\mathbb{R}^d\setminus \{0\}$. The Fourier transform
of $f\in L^1(\mathbb{R}^d)$ is given by $\widehat{f}(\xi)=\int f(x)
e(\langle x, \xi\rangle) \, dx$.

We fix $\chi_0$ to be a smooth cut-off function whose value is $1$
on $B(0, 1/2)$ and $0$ on the complement of $B(0, 1)$. For a set
$E\subset \mathbb{R}^d$ and a positive number $a$, we define
$E_{(a)}$ to be the larger set
\begin{equation*}
E_{(a)}=\{x\in \mathbb{R}^d: \dist(E, x)<a \}.
\end{equation*}

We use the differential operators
\begin{equation*}
D^{\nu}_{x}=\frac{\partial^{|\nu|}}{\partial x_1^{\nu_1} \cdots
\partial x_d^{\nu_d}} \quad \big(\nu=(\nu_1, \ldots, \nu_d)\in
\mathbb{N}_0^d, |\nu|=\sum_{i=1}^{d}\nu_i \big)
\end{equation*}
and the gradient operator $\nabla_x$. We often omit the subscript if
no ambiguity occurs.

{\it Structure of the paper:}  We first establish some preliminaries
in \S 2-4 mainly for compact convex domains with no curvature
assumption on the smooth boundary. We then give a proof of Theorem
\ref{Rd-finitetype} in \S \ref{mainterm}, in which the problem is
reduced to the estimate of two sums (Sum I and II). The estimate of
Sum I that we give essentially works for general compact convex
domains, while the curvature condition on the boundary is used in
the estimate of Sum II. Since it is easy to modify the proof of
Theorem \ref{Rd-finitetype} to prove the other theorems, we only
provide brief proofs of Theorem \ref{Rd-varchenkoplus} in \S
\ref{mainterm} and Theorem \ref{R2-general} and \ref{R2-finitetype}
in \S \ref{R2-section}. At last we collect some standard analysis
results in Appendix \ref{app1} and prove an estimate of exponential
sums in Appendix \ref{app2} for interested readers.

\section{Some Geometric Facts}\label{geometry}

Assume $\mathcal{B}\subset \mathbb{R}^d$ ($d\geqslant 2$) is a
compact convex domain and its boundary is a smooth hypersurface
(curve if $d=2$). For a point $x\in
\partial \mathcal{B}$, let $K(x)$ be the (Gaussian) curvature of $\partial \mathcal{B}$ at $x$. Define
\begin{equation*}
(\partial \mathcal{B})_+=\{x\in \partial \mathcal{B}: K(x)>0\} \quad
\textrm{and} \quad (\partial \mathcal{B})_0=\{x\in \partial
\mathcal{B}: K(x)=0\},
\end{equation*}
thus
\begin{equation*}
\partial \mathcal{B}=(\partial \mathcal{B})_+ \biguplus (\partial
\mathcal{B})_0.
\end{equation*}

The Gauss map of $\partial \mathcal{B}$, denoted by $\vec{n}$, maps
each boundary point $x\in \partial \mathcal{B}$ to a unit exterior
normal $\vec{n}(x)\in S^{d-1}$. Then
\begin{equation*}
S^{d-1}=\vec{n}((\partial \mathcal{B})_+)\biguplus \vec{n}((\partial
\mathcal{B})_0).
\end{equation*}
Note that the restriction of $\vec{n}$ to $(\partial
\mathcal{B})_+$, namely
\begin{equation*}
\vec{n}|(\partial \mathcal{B})_+ : (\partial \mathcal{B})_+
\rightarrow \vec{n}((\partial \mathcal{B})_+)\subset S^{d-1},
\end{equation*}
is bijective. For $\xi\neq 0$ with $\xi/|\xi|\in \vec{n}((\partial
\mathcal{B})_+)$ let $x(\xi):=\vec{n}^{-1}(\xi/|\xi|)$ be the unique
point on $\partial \mathcal{B}$ where the unit exterior normal is
$\xi/|\xi|$. Hence $K_{\xi}=K(x(\xi))$ is well defined for such
points $\xi$.

For $\xi\neq 0$ with $\xi/|\xi|\in \vec{n}((\partial
\mathcal{B}_{\theta})_+)$ let $x^{\theta}(\xi)=\theta
x(\theta^{t}\xi)$ and $K_{\xi}^{\theta}=K_{\theta^{t}\xi}$. Then
$x^{\theta}(\xi)$ is the unique point on $\partial
\mathcal{B}_{\theta}$ where the exterior normal is $\xi$ and
$K_{\xi}^{\theta}$ is the curvature of $\partial
\mathcal{B}_{\theta}$ at $x^{\theta}(\xi)$.

\begin{lemma}\label{curv-in-ball}
Assume $\mathcal{B}\subset \mathbb{R}^d$ ($d\geqslant 2$) is a
compact convex domain and its boundary is a smooth hypersurface
(curve if $d=2$). Then there exists a constant $c_1>0$ (depending
only on $\mathcal{B}$) such that, for any $\xi \in \vec{n}((\partial
\mathcal{B})_+)$, if $\eta\in B(\xi, c_1(K_\xi)^{2})\subset
\mathbb{R}^d$ then $\eta/|\eta|\in \vec{n}((\partial
\mathcal{B})_+)$ and
\begin{equation*}
K_\xi/2\leqslant K_\eta \leqslant 3K_\xi/2.
\end{equation*}
\end{lemma}
\begin{proof}
For any $\xi \in \vec{n}((\partial \mathcal{B})_+)$ it follows from
the mean value theorem that there exists a constant $c$ (depending
only on $\mathcal{B}$) such that
\begin{equation*}
K_{\xi}/2\leqslant K(y)\leqslant 3K_{\xi}/2 \quad \textrm{if $y\in
B(x(\xi), cK_{\xi})\cap \partial\mathcal{B}$}.
\end{equation*}

It is a consequence of Lemma \ref{app:lemma:1} that the Gauss map is
bijective from a subset of $B(x(\xi),
cK_{\xi})\cap\partial\mathcal{B}$ onto a subset of $S^{d-1}$
containing $B(\xi, c'(K_{\xi})^2)\cap S^{d-1}$ where the constant
$c'$ depends only on $\mathcal{B}$. Then the lemma follows easily.
\end{proof}

\begin{lemma}\label{s2-lemma2}
Assume $\mathcal{B}\subset \mathbb{R}^d$ ($d\geqslant 2$) is a
compact convex domain and its boundary is a smooth hypersurface
(curve if $d=2$). Then
\begin{equation*}
|\vec{n}(\{x\in \partial \mathcal{B}: K(x)< \delta \})|\leqslant
C_{\mathcal{B}} \delta |\{x\in \partial \mathcal{B}: 0<K(x)< \delta
\}|
\end{equation*}
where the absolute value denotes the induced Lebesgue measure on
surfaces (curves if $d=2$).
\end{lemma}

\begin{proof} Note that
\begin{equation*}
\{x\in \partial \mathcal{B}: K(x)< \delta \}=\{x\in \partial
\mathcal{B}: 0<K(x)<\delta \}\biguplus (\partial \mathcal{B})_0.
\end{equation*}
We first have
\begin{equation*}
|\vec{n}((\partial \mathcal{B})_0)|=0
\end{equation*}
due to Sard's theorem (see Lang~\cite[p.~286]{lang}). Hence it
suffices to prove
\begin{equation*}
|\vec{n}(\{x\in \partial \mathcal{B}: 0<K(x)< \delta \})|\leqslant
C_{\mathcal{B}} \delta |\{x\in \partial \mathcal{B}: 0<K(x)< \delta
\}|.
\end{equation*}
By using a standard technique found in the proof of certain covering
lemma of Vitali type (see Stein~\cite{stein}), we reduce the above
estimate to
\begin{equation*}
|\vec{n}(B)|\leqslant C_{\mathcal{B}} \delta |B|,
\end{equation*}
where $B\subset \{x\in \partial \mathcal{B}: 0<K(x)< \delta \}$ is a
ball in $\partial \mathcal{B}$. However this last estimate follows
from the equality $d\sigma=K(x)dA$ where $dA$ is the volume element
of $\partial \mathcal{B}$ at the point $x\in
\partial \mathcal{B}$ and $d\sigma$ the volume element of $S^{d-1}$
at the point $\vec{n}(x)\in S^{d-1}$ (see \cite[p.~47]{geometryI};
this equality can also be verified by using local coordinate
charts). This finishes the proof.
\end{proof}

We say that the boundary $\partial \mathcal{B}$ is of finite type if
at every point $x\in
\partial \mathcal{B}$, every one dimensional tangent line to
$\partial \mathcal{B}$ at $x$ makes finite order of contact with
$\partial \mathcal{B}$. If $\partial \mathcal{B}$ is of finite type,
the maximum order of contact over all $x\in
\partial \mathcal{B}$ and all tangent lines to $x\in \partial \mathcal{B}$
is called the type of $\partial \mathcal{B}$.

We will always assume below that the type is $\geqslant 3$ since if
the type is two then we recover the case of nonvanishing (Gaussian)
curvature.

\begin{lemma}\label{s2-lemma3}
Assume $\mathcal{B}\subset \mathbb{R}^d$  ($d\geqslant 2$) is a
compact convex domain and its boundary is a smooth hypersurface
(curve if $d=2$) of finite type $\omega$. Then
\begin{equation*}
|\{x\in \partial \mathcal{B}: K(x)<\delta\}|\leqslant
C_{\mathcal{B}} \delta^{1/(d-1)(\omega-2)}.
\end{equation*}
\end{lemma}

\begin{proof}
By using a compactness argument and local coordinates we may only
regard $K$ as a function of $x'$ in a neighborhood $B(0, C_0)$ of
$0$ in $\mathbb{R}^{d-1}$ for some constant $C_0$. We may assume
that $K$, $\partial K/\partial x_1$, \ldots, $\partial^h K/\partial
x_1^h$ (with $h=(d-1)(\omega-2)$) do not vanish simultaneously (see
Svensson~\cite[p.~19]{svensson}). We then apply Svensson's
\cite[Lemma 3.3]{svensson} to $K$ in $x_1$-direction, which yields
\begin{equation*}
|\{x_1 : |x_1|\leqslant C_0, K(x')<\delta\}|\leqslant
C_{\mathcal{B}} \delta^{1/h},
\end{equation*}
and the trivial estimate in $x_2$, \ldots, $x_{d-1}$-directions.
Thus the desired estimate follows.
\end{proof}

\section{Nonvanishing $d\times d$ Determinants}\label{non-vanishing}

In this section we always assume that $\mathcal{B}\subset
\mathbb{R}^d$  ($d\geqslant 2$) is a compact convex domain and its
boundary is a smooth hypersurface (curve if $d=2$).

The support function of $\mathcal{B}$ is given by $H(\xi)=\sup_{y\in
\mathcal{B}}\langle \xi, y\rangle$ for any nonzero $\xi\in
\mathbb{R}^d$. In particular $H(\xi)=\langle \xi, x(\xi)\rangle$ for
any nonzero $\xi$ with $\xi/|\xi|\in \vec{n}((\partial
\mathcal{B})_+)$. It is positively homogeneous of degree one, {\it
i.e.} $H(\lambda \xi)=\lambda H(\xi)$ if $\lambda>0$.

The next two lemmas can be easily proved by using local coordinates,
hence we omit the proof.

\begin{lemma} $H$ is smooth at every $\xi\in \vec{n}((\partial
\mathcal{B})_+)$ and satisfies
\begin{equation*}
D^{\nu}H(\xi) \lesssim 1 \quad \textrm{for $0\leqslant
|\nu|\leqslant 1$}
\end{equation*}
and
\begin{equation*}
D^{\nu}H(\xi) \lesssim (K_\xi)^{3-2|\nu |} \quad \textrm{for } |\nu
|\geqslant 2,
\end{equation*}
where implicit constants may depend only on $|\nu|$ and
$\mathcal{B}$.
\end{lemma}

\begin{remark}\label{remark1}
For $\theta\in SO(d)$, we will denote the support function of
$\mathcal{B}_{\theta}$ by $H_{\theta}(\xi)=\sup_{y\in
\mathcal{B}_{\theta}}\langle \xi, y\rangle$. Since
$H_{\theta}(\xi)=H(\theta^{t}\xi)$, we can easily get bounds for
$H_{\theta}$ in the same form as in the above lemma (with
$\vec{n}((\partial \mathcal{B})_+)$ and $K_\xi$ replaced by
$\theta\vec{n}((\partial \mathcal{B})_+)$ and $K_\xi^{\theta}$
respectively).
\end{remark}

\begin{lemma}\label{lemma4:2}
Assume that $\xi\neq 0$ with $\xi/|\xi|\in \vec{n}((\partial
\mathcal{B})_+)$. If $d\geqslant 3$ the eigenvalues of the matrix
$\nabla^2_{\xi \xi}H(\xi)$ are $0$, $(|\xi|\kappa_1)^{-1}$, \ldots,
$(|\xi|\kappa_{d-1})^{-1}$, where $\{\kappa_j\}_{j=1}^{d-1}$ are the
principle curvatures of $\partial \mathcal{B}$ at $x(\xi)$; if $d=2$
the eigenvalues are $0$ and $(|\xi|K_{\xi})^{-1}$.
\end{lemma}

Given $d$ vectors $v_1, \ldots, v_d\in \mathbb{R}^d$, by writing
$V=(v_1, \ldots, v_d)$ we mean $V$ is the matrix in
$\mathbb{R}^{d\times d}$ with column vectors $v_1, \ldots, v_d$. If
$y\neq 0$ we define $F_{\theta}(u_1, \ldots,
u_d)=H_{\theta}(y+\sum_{l=1}^{d}u_l v_l)$, $u_l\in \mathbb{R}$
$(l=l, \ldots, d)$. For $q\in \mathbb{N}$ let
\begin{equation*}
h_q^{\theta}(y, v_1, \ldots, v_d)=\det\left(g_{i, j}^{\theta}(y,
v_1, \ldots, v_d) \right)_{1\leqslant i, j\leqslant d},
\end{equation*}
where
\begin{equation*}
g_{i, j}^{\theta}(y, v_1, \ldots,
v_d)=\frac{\partial^{q+2}F_{\theta}}{\partial u_1
\partial u_i \partial u_j \partial u_d^{q-1}}(0).
\end{equation*}

The following lemma is a higher dimensional analogue of Guo's
\cite[Lemma 3.4]{guo2}, which enables us to apply the method of
stationary phase later in the estimate of certain exponential sums.
The result is for $\xi\in \vec{n}((\partial
\mathcal{B}_{\theta})_+)$, but can be easily extended to $\xi\neq 0$
with $\xi/|\xi|\in \vec{n}((\partial \mathcal{B}_{\theta})_+)$ by
using the homogeneity of $H_{\theta}$. We will follow M\"{u}ller's
method used to prove his \cite[Lemma 3]{mullerII}.

\begin{lemma} \label{non-vanishinglemma}
If $d\geqslant 3$, for every $\xi\in \vec{n}((\partial
\mathcal{B}_{\theta})_+)$ there exist $d$ linearly independent
vectors $v_l=v_l^{\theta}(\xi)\in \mathbb{Z}^d$ ($l=1, \ldots, d$)
such that
\begin{equation}
\begin{split}
&|v_1|\asymp (K_{\xi}^{\theta})^{-d-2q-8+1/(d-1)},\\
&|v_l|\asymp
(K_{\xi}^{\theta})^{-d-2q-5+1/(d-1)} \quad (l=2, \ldots, d), \\
&|\det (V)|\asymp (K_{\xi}^{\theta})^{d(-d-2q-5+1/(d-1))}, \\
&\|V^{-1}\|\lesssim (K_{\xi}^{\theta})^{d+2q+2-1/(d-1)},
\end{split}\label{lemma4:4-1}
\end{equation}
where $V=(v_1,\ldots, v_d)$. Furthermore there exists a constant
$c_2>0$ (depending only on $q$ and $\mathcal{B}$) such that, for
$\eta \in B(\xi, c_2(K_\xi^{\theta})^{d+2q+7-1/(d-1)})$,
\begin{equation}
|h_q^{\theta}(\eta, v_1,\ldots, v_d)|\gtrsim
(K_\xi^{\theta})^{(-d-2q-5+1/(d-1))d(q+2)-3d+5-1/(d-1)}. \label{bbb}
\end{equation}

The constants implicit in \eqref{lemma4:4-1} and \eqref{bbb} depend
only on $q$ and $\mathcal{B}$.
\end{lemma}

\begin{proof}
Let $\xi\in \vec{n}((\partial \mathcal{B}_{\theta})_+)$ be
arbitrarily fixed.

{\bf Step 1.} Let $p_1=\xi$. We first choose $d-1$ vectors $p_2,
\ldots, p_{d}\in S^{d-1}$ such that $P=(p_1, \ldots, p_d)\in
\mathbb{R}^{d\times d}$ is an orthogonal matrix. Let
$\widetilde{H}_{\theta}(y)=H_{\theta}(Py)$. Then
$\widetilde{H}_{\theta}$ is positively homogeneous of degree one and
smooth at $e_1$. Since the matrix $\nabla^2
\widetilde{H}_{\theta}(e_1)$ is similar to $\nabla^2
H_{\theta}(\xi)$ it follows from Lemma \ref{lemma4:2} that the
eigenvalues of $\nabla^2 \widetilde{H}_{\theta}(e_1)$ are $0$,
$\beta_1, \ldots, \beta_{d-1}$, where $\{\beta_j^{-1}\}_{j=1}^{d-1}$
are the principle curvatures of $\partial \mathcal{B}_{\theta}$ at
$x^{\theta}(\xi)$. Without loss of generality we assume
$\beta_1=\max_{1\leqslant j\leqslant d-1} \beta_j$, therefore
$\beta_1\geqslant (K_\xi^{\theta})^{-1/(d-1)}$.

Set $A=\nabla^2 \widetilde{H}_{\theta}(e_1)$. $A$ is a symmetric
matrix of rank $d-1$ with vanishing first row and column (due to the
homogeneity of $\widetilde{H}_{\theta}$; see the proof of
M\"{u}ller's \cite[Lemma 3]{mullerII}). Choose a system of
orthonormal eigenvectors $w_1', \ldots, w_{d-1}'$ of $A$, whose
first components vanish, such that the eigenvalue of $w_j'$ is
$\beta_j$. For $\alpha>1$ denote
\begin{equation*}
w_l=\left\{
\begin{array}{ll}
w_1'+\alpha e_1   &\quad \textrm{if $l=1$}, \\
w_l'              &\quad \textrm{if $2\leqslant l\leqslant d-1$},\\
e_1              &\quad \textrm{if $l=d$}.
\end{array} \right.
\end{equation*}
Then $A w_l=\beta_l w_l'$ ($l=1, \ldots, d-1$) and $w_1$ is
orthogonal to $w_l'$ ($l=2, \ldots, d-1$). We also have $|w_1|\asymp
\alpha$, $|w_l|=1$ ($l=2,\ldots, d$), and $|\det(W)|=1$ where
$W=(w_1, \ldots, w_d)$. Let $v_l^*=Pw_l$. Then $|v_1^*|\asymp
\alpha$, $|v_l^*|=1$ ($l=2, \ldots, d$), and $|\det(V^*)|=1$ where
$V^*=(v_1^*, \ldots, v_d^*)$. We claim that if $\alpha=C_{q,
\mathcal{B}} (K_{\xi}^{\theta})^{-3}$ with a sufficiently large
$C_{q, \mathcal{B}}$ then
\begin{equation}
|h_q^{\theta}(\xi, v_1^*,\ldots, v_d^*)|\gtrsim
(K_\xi^{\theta})^{-3d+5-1/(d-1)} \label{step1}
\end{equation}
with $F_{\theta}(u_1, \ldots, u_d)=H_{\theta}(\xi+\sum_{l=1}^{d}u_l
v_l^*)$.

This claim can be proved by a straightforward computation (given
below). Note that $F_{\theta}(u_1, \ldots,
u_d)=\widetilde{H}_{\theta}(e_1+\sum_{l=1}^{d}u_l w_l)$ and we will
use this formula to compute $g_{i, j}^{\theta}(\xi, v_1^{*}, \ldots,
v_d^{*})=:b_{i,j}^{\theta}(\alpha)$.  If $1\leqslant i, j\leqslant
d-1$,
\begin{equation}
b_{i,j}^{\theta}(0)=(\nabla\cdot w_1')(\nabla\cdot w_i')(\nabla\cdot
w_j') \partial_{y_1}^{q-1}\widetilde{H}_{\theta}(e_1)\lesssim
(K_\xi^{\theta})^{-3}. \label{bound_bij0}
\end{equation}
The last inequality is due to the homogeneity of
$\widetilde{H}_{\theta}$ (see the proof of M\"{u}ller's \cite[Lemma
3]{mullerII}) and Remark \ref{remark1}.

If $i=1$, $1\leqslant j\leqslant d-1$, then
\begin{equation}
b_{1,j}^{\theta}(\alpha)=b_{1,j}^{\theta}(0)+3\alpha (-1)^q
q!\beta_1 \delta_{1j},\label{b1j}
\end{equation}
where $\delta_{ij}$ is the Kronecker notation.

If $2\leqslant i, j\leqslant d-1$, then
\begin{equation}
b_{i,j}^{\theta}(\alpha)=b_{i,j}^{\theta}(0)+\alpha (-1)^q q!\beta_j
\delta_{ij}.\label{bij}
\end{equation}

If $1\leqslant i\leqslant d$, $j=d$, then
\begin{equation}
b_{i,d}^{\theta}(\alpha)=(-1)^q q!\beta_1 \delta_{1i}.\label{bid}
\end{equation}

Using formulas \eqref{bij} and \eqref{bid}, we get
\begin{equation*}
\begin{split}
|h_q^{\theta}(\xi, v_1^*,\ldots, v_d^*)|&=(q!\beta_1)^2
|\det(b_{i,j}^{\theta}(\alpha))_{2\leqslant i, j\leqslant d-1}|\\
        &=(q!\beta_1)^2
|\det(b_{i,j}^{\theta}(0)+\alpha (-1)^q q!\beta_j
\delta_{ij})_{2\leqslant i,
j\leqslant d-1}|\\
&=\beta_1(K_\xi^{\theta})^{-3d+5}|q!^d C_{q,
\mathcal{B}}^{d-2}+O(C_{q, \mathcal{B}}^{d-3})|
\end{split}
\end{equation*}
where we have used \eqref{bound_bij0}, $\beta_j\gtrsim 1$, and
$\prod \beta_j=(K_\xi^{\theta})^{-1}$ to get the last equality.
Since $\beta_1\geqslant (K_\xi^{\theta})^{-1/(d-1)}$, we get
\eqref{step1} if $C_{q, \mathcal{B}}$ is sufficiently large.

{\bf Step 2.} For any $N\in \mathbb{N}$, there exist $v_l\in
\mathbb{Z}^d$ ($l=1, \ldots, d$) such that
$|v_l^{**}-v_l^{*}|\leqslant \sqrt{d}/N$ where $v_{l}^{**}=v_l/N$.
If $N\geqslant C(K_\xi^{\theta})^{-3}$ then $|v_1^{**}|\asymp
(K_\xi^{\theta})^{-3}$, $|v_l^{**}|\asymp 1$ ($l=2, \ldots, d$), and
$|\det(V^{**})|\asymp 1$ where $V^{**}=(v_1^{**}, \ldots,
v_d^{**})$.

Assume $N$ is the smallest integer not less than
$C'(K_\xi^{\theta})^{-d-2q-5+1/(d-1)}$ with $C'$ chosen below and
$\eta\in B(\xi, c_2 r^{\theta}(\xi))$ with
$r^{\theta}(\xi)=(K_\xi^{\theta})^{d+2q+7-1/(d-1)}$ and
$c_2\leqslant c_1$, where $c_1$ is the constant appearing in Lemma
\ref{curv-in-ball}. By the mean value theorem, Lemma
\ref{curv-in-ball}, and Remark \ref{remark1}, we get
\begin{align*}
&|g_{i, j}^{\theta}(\xi, v_1^*,\ldots, v_d^*)-g_{i,
j}^{\theta}(\eta,
v_1^{**}, \ldots,  v_d^{**})|\\
&\quad \lesssim \left\{ \begin{array}{lll}
(K_\xi^{\theta})^{-2q-10}(N^{-1}+c_2(K_\xi^{\theta})^{-2}r^{\theta}(\xi))
& \textrm{if
$i=j=1$},\\
(K_\xi^{\theta})^{-2q-7}(N^{-1}+c_2(K_\xi^{\theta})^{-2}r^{\theta}(\xi)) & \textrm{if $i=1, j\geqslant 2$}, \\
(K_\xi^{\theta})^{-2q-4}(N^{-1}+c_2(K_\xi^{\theta})^{-2}r^{\theta}(\xi))
 & \textrm{if $i\geqslant 2, j \geqslant 2$}.
\end{array}\right.
\end{align*}
These estimates, together with the bounds of $g_{i, j}^{\theta}(\xi,
v_1^*,\ldots, v_d^*)$'s (given by \eqref{bound_bij0}, \eqref{b1j},
\eqref{bij}, and \eqref{bid}), lead to
\begin{equation*}
|h_q^{\theta}(\xi, v_1^*,\ldots, v_d^*)-h_q^{\theta}(\eta, v_1^{**},
\ldots,  v_d^{**})|\lesssim
(K_\xi^{\theta})^{-4d-2q}(N^{-1}+c_2(K_\xi^{\theta})^{-2}r^{\theta}(\xi)).
\end{equation*}

If $C'$ is sufficiently large and $c_2$ is sufficiently small, it
then follows from \eqref{step1} that
\begin{equation*}
|h_q^{\theta}(\eta, v_1^{**}, \ldots,  v_d^{**})|\gtrsim
(K_\xi^{\theta})^{-3d+5-1/(d-1)}.
\end{equation*}

The desired estimates now follow from the following two equalities:
\begin{equation*}
|h_q^{\theta}(\eta, v_1, \ldots, v_d)|=N^{d(q+2)}|h_q^{\theta}(\eta,
v_1^{**}, \ldots,  v_d^{**})|
\end{equation*}
and
\begin{equation*}
V^{-1}=N^{-1}(\textrm{adjugate matrix of $V^{**}$})/\det(V^{**}).
\end{equation*}
\end{proof}

For $d=2$ case Guo's \cite[Lemma 3.4]{guo2} gives a similar result
but in a nicer form. That lemma can be proved by using the same
method. In particular, the bound $g_{11}(\xi, v_1, v_2)\lesssim
(K_{\xi}^{\theta})^{-2q-1}$ is used in its proof, but later we find
a better bound of $g_{11}$, namely $g_{11}(\xi, v_1, v_2)\lesssim
(K_{\xi}^{\theta})^{-3}$ (just like the bound \eqref{bound_bij0} in
the above proof). By using the latter bound without modifying too
much of the proof of Guo's \cite[Lemma 3.4]{guo2}, we are able to
prove the following improved result, which eventually leads to our
estimates in Theorem \ref{R2-general} and \ref{R2-finitetype}.

\begin{lemma}\label{r2-keylemma} If $d=2$, for every $\xi\in
\vec{n}((\partial \mathcal{B}_{\theta})_+)$ there exist two
orthogonal vectors $v_i=v_i(\xi)\in \mathbb{Z}^2$ ($i=1, 2$) such
that
\begin{equation}
|v_1|=|v_2|\asymp (K_{\xi}^{\theta})^{-2q-2}\quad \textrm{and}\quad
\|V^{-1}\|\lesssim (K_{\xi}^{\theta})^{2q+2}, \label{s9-lemma4:4-1}
\end{equation}
where $V=(v_1, v_2)$. Furthermore there exists a constant $c_2>0$
(depending only on $q$ and $\mathcal{B}$) such that, for $\eta \in
B(\xi, c_2(K_\xi^{\theta})^{2q+4})$,
\begin{equation}
|h_q^{\theta}(\eta, v_1, v_2)|\gtrsim
(K_\xi^{\theta})^{-4q^2-12q-10}. \label{s9-bbb}
\end{equation}

The constants implicit in \eqref{s9-lemma4:4-1} and \eqref{s9-bbb}
depend only on $q$ and $\mathcal{B}$.
\end{lemma}


\section{The Fourier Transform of Certain Indicator Functions}\label{four-tran}

In this section we will establish an asymptotic formula of the
Fourier transform of the indicator function $\chi_{\mathcal{B}}$ for
convex domains $\mathcal{B}$ in $\mathbb{R}^d$, which generalizes
the results in Guo's \cite[Section 4]{guo2}.

\begin{lemma}\label{monotonicity}
Assume $\mathcal{B}\subset \mathbb{R}^d$ ($d\geqslant 2$) is a
compact convex domain and its boundary is a smooth hypersurface
(curve if $d=2$). Then there exist two positive constants $c$ and
$c_3$ (both depending only on $\mathcal{B}$) such that, for any $\xi
\in \vec{n}((\partial \mathcal{B})_+)\cap (-\vec{n}((\partial
\mathcal{B})_+))$ and $r\leqslant c_3$,
\begin{equation}
|\langle \vec{n}(x), \xi\rangle|\leqslant 1-c r^2(\min(K_{\xi},
K_{-\xi}))^4 \label{ineq-mono}
\end{equation}
if $x$ is in $\partial \mathcal{B}\setminus (B(x(\xi), r
K_{\xi})\cup B(x(-\xi), r K_{-\xi}))$.
\end{lemma}

\begin{proof} It follows from Lemma \ref{app:lemma:1} that there exists a constant
$c_3>0$ (depending only on $\mathcal{B}$) such that, for any
$r\leqslant c_3$, the Gauss map is bijective from $B(x(\xi), r
K_{\xi})\cap \partial\mathcal{B}$ and $B(x(-\xi), r
K_{-\xi})\cap\partial\mathcal{B}$ to two subsets of $S^{d-1}$ which
contain $B(\xi, c' r(K_{\xi})^2)\cap S^{d-1}$ and $B(-\xi, c'
r(K_{-\xi})^2)\cap S^{d-1}$ respectively where the constant $c'>0$
depends only on $\mathcal{B}$. Then the lemma follows easily with
$c=2c'^2/\pi^2$.
\end{proof}

\begin{theorem} \label{asym-R^d}
Assume $\mathcal{B}\subset \mathbb{R}^d$ ($d\geqslant 2$) is a
compact convex domain and its boundary is a smooth hypersurface
(curve if $d=2$). Let $n_l$ ($l=1, \ldots, d$) be the
$l^{\textrm{th}}$ component of the Gauss map of $\partial
\mathcal{B}$ and $dS$ the induced Lebesgue measure on surfaces
(curves if $d=2$). For any $\xi \in \vec{n}((\partial
\mathcal{B})_+)\cap (-\vec{n}((\partial \mathcal{B})_+))$ we have
\begin{equation*}
\begin{split}
\widehat{n_l dS}(\lambda \xi)&=(e^{\pi i(d-1)/4} \xi_l
(K_{\xi})^{-1/2}e^{-2\pi i\lambda
H(\xi)}\\
&\quad +e^{-\pi i(d-1)/4}(-\xi_l)(K_{-\xi})^{-1/2}e^{2\pi i\lambda H(-\xi)})\lambda^{-(d-1)/2}\\
&\quad\quad
+O(\lambda^{-(d+1)/2}\delta^{-(d+5)/2}+\lambda^{-N}\delta^{-4N})
\quad \textrm{for $\lambda >0$},
\end{split}
\end{equation*}
where $H(\xi)=\sup_{y\in \mathcal{B}} \langle y, \xi\rangle$, $N\in
\mathbb{N}$, and $\delta=\min(K_{\xi}, K_{-\xi})$. The implicit
constant depends only on $N$ and $\mathcal{B}$.
\end{theorem}

\begin{proof}
We will only prove the case $d\geqslant 3$ below while the case
$d=2$ is easier and can be handled in the same way. Note that there
exists a $C_0>0$ such that, for any $x\in
\partial \mathcal{B}$, the boundary $\partial \mathcal{B}$ in a
neighborhood of $x$ can be parametrized by
\begin{equation}
\begin{split}
\vec{r}(u, x)=x+&\sum_{j=1}^{d-1}u_j \vec{t}_j(x)+h(u,
x)(-\vec{n}(x)) , \\
&\textrm{for $u=(u_1, \ldots, u_{d-1})\in B_0=B(0, C_0)\subset
\mathbb{R}^{d-1}$},
\end{split}\label{parametrization}
\end{equation}
where $\{\vec{t}_j(x)\}_1^{d-1}$ is an orthonormal basis of the
tangent plane of $\partial \mathcal{B}$ at $x$ (we require that the
basis $\{ \vec{t}_1(x), \ldots, \vec{t}_{d-1}(x), -\vec{n}(x)\}$ has
the same orientation as $\{e_1, \ldots, e_d \}$) and $h(\cdot \, ,
x)\in C^{\infty}(B_0)$ such that $h(0, x)=0$, $\nabla_{u}h(0, x)=0$,
and $\det\nabla^2_{uu}h(0, x)=K(x)$.

For any fixed $\xi \in \vec{n}((\partial \mathcal{B})_+)\cap
(-\vec{n}((\partial \mathcal{B})_+))$ decompose $n_l$ as a sum
\begin{equation*}
n_l=\psi_1+\psi_2+\psi_3
\end{equation*}
where
\begin{equation*}
\psi_1(x, \xi)=n_l(x)\chi_0(\frac{x-x(\xi)}{c_4 K_{\xi}}) \textrm{
and } \psi_2(x, \xi)=n_l(x)\chi_0(\frac{x-x(-\xi)}{c_4 K_{-\xi}}),
\end{equation*}
where $c_4>0$ is determined below and $\chi_0$ is the fixed cut-off
function (see \S \ref{introduction}).

We first estimate $\widehat{\psi_1 dS}$ (while $\widehat{\psi_2 dS}$
is handled in the same way). Applying the parametrization
\eqref{parametrization} at $x(\xi)$ yields
\begin{equation}
\widehat{\psi_1 dS}(\lambda \xi)=e^{-2\pi i\lambda \langle \xi,
x(\xi) \rangle} \int \tau(u, \xi) e^{2\pi i\lambda h(u, x(\xi))}
\,du, \label{1stpart}
\end{equation}
where $\tau(u, \xi)=\psi_1(\vec{r}(u, x(\xi)), \xi)(1+|\nabla_u h(u,
x(\xi))|^2)^{1/2}$ such that
\begin{equation*}
\tau(\cdot \, , \xi)\in  C^{\infty}_{c}(B(0, c_4 K_{\xi}))
\end{equation*}
and
\begin{equation*}
|D^{\nu}_{u}\tau(u, \xi) |\leqslant C(c_4 K_{\xi})^{-|\nu|}.
\end{equation*}

By a change of variable the integral in \eqref{1stpart}, denoted by
$\Delta(\xi)$, is
\begin{equation*}
\Delta(\xi)=K_{\xi}^{d-1} \int \tau(K_{\xi}u, \xi) e^{2\pi i\lambda
h(K_{\xi}u, x(\xi))} \,du.
\end{equation*}
Applying a quantitative version of the Morse Lemma (see the proof of
Sogge and Stein's \cite[Lemma 2]{soggestein}) we can find an
$\alpha_1>0$ and a smooth invertible mapping $u\mapsto v$ from $B(0,
\alpha_1)$ to a neighborhood of the origin in $v$-space, so that
$|D^{\nu}_{u}v|\leqslant C$, $|D^{\nu}_v u|\leqslant C$,
$\det(\nabla_v u(0))=1$, and
\begin{equation*}
h(K_{\xi}u, x(\xi))=K_{\xi}^2(\mu_1 v_1^2+\ldots+\mu_{d-1}
v_{d-1}^2)/2, \quad u\in B(0, \alpha_1),
\end{equation*}
where $\mu_1, \ldots, \mu_{d-1}$ are the eigenvalues of the matrix
$\nabla^2_{uu}h(0, x(\xi))$. Let $c_4\leqslant \alpha_1$. Then
\begin{equation*}
\Delta(\xi)=K_{\xi}^{d-1} \int \tilde{\tau}(v, \xi)
e^{i\tilde{\lambda}(\mu_1 v_1^2+\ldots+\mu_{d-1} v_{d-1}^2)/2} \,
dv,
\end{equation*}
where $\tilde{\lambda}=2\pi \lambda K_{\xi}^2$ and $\tilde{\tau}(v,
\xi)=\tau(K_{\xi}u(v), \xi)|\textrm{det}(\nabla_v u)|$. Applying
Lemma \ref{app:lemma:hor} to the integral above yields an asymptotic
expansion, which in turn gives
\begin{equation*}
\begin{split}
\widehat{\psi_1 dS}(\lambda \xi)&=e^{\pi i(d-1)/4}\xi_l
(K_{\xi})^{-1/2}e^{-2\pi i\lambda
H(\xi)}\lambda^{-(d-1)/2}\\
&\quad +O(\lambda^{-(d+1)/2}K_{\xi}^{-(d+5)/2}).
\end{split}
\end{equation*}

The estimate $\widehat{\psi_3 dS}=O(\lambda^{-N}\delta^{-4N})$
follows from Lemma \ref{monotonicity} and integration by parts (see
Stein~\cite[p.~350]{stein} for a similar argument). This finishes
the proof.
\end{proof}

As a consequence of the Gauss--Green formula we get:

\begin{corollary} \label{s4-asym}
Assume $\mathcal{B}\subset \mathbb{R}^d$ ($d\geqslant 2$) is a
compact convex domain and its boundary is a smooth hypersurface
(curve if $d=2$). For any $\xi \in \vec{n}((\partial
\mathcal{B})_+)\cap (-\vec{n}((\partial \mathcal{B})_+))$ we have
\begin{equation*}
\begin{split}
\widehat{\chi}_{\mathcal{B}}(\lambda \xi)&=((2\pi)^{-1}e^{\pi
i(d+1)/4} (K_{\xi})^{-1/2}e^{-2\pi i\lambda
H(\xi)}\\
&\quad +(2\pi)^{-1}e^{-\pi i(d+1)/4}(K_{-\xi})^{-1/2}e^{2\pi i\lambda H(-\xi)})\lambda^{-(d+1)/2}\\
&\quad\quad
+O(\lambda^{-(d+3)/2}\delta^{-(d+5)/2}+\lambda^{-N-1}\delta^{-4N})
\quad\quad \textrm{for $\lambda >0$},
\end{split}
\end{equation*}
where $H(\xi)=\sup_{y\in \mathcal{B}} \langle y, \xi\rangle$, $N\in
\mathbb{N}$, and $\delta=\min(K_{\xi}, K_{-\xi})$. The implicit
constant depends only on $N$ and $\mathcal{B}$.
\end{corollary}

\section{The $\mathbb{R}^d$ ($d\geqslant 3$) case}\label{mainterm}

By a very standard argument, Theorem \ref{Rd-finitetype} follows
easily from the following lemma (see Guo~\cite[p.~18]{guo2},
Iosevich~\cite[p.~26-27]{iosevich}, or Iosevich, Sawyer, and
Seeger~\cite[p.~168-169]{I-S-S-0} for this argument).

\begin{lemma}\label{lemma5:1}
Let $\mathcal{B}\subset \mathbb{R}^d$ ($d\geqslant 3$) be a compact
convex domain and $\rho\in C_0^{\infty}(\mathbb{R}^d)$ such that
$\int_{\mathbb{R}^d}\rho(y)\,dy=1$. If the boundary is a smooth
hypersurface of finite type $\omega$ then, for $j\in \mathbb{N}$, we
have
\begin{equation}
\int_{SO(d)} \sup_{2^{j-1}\leqslant t\leqslant 2^{j+2}}
t^{2d/(d+1)+\zeta_d+\sigma(d, \omega)}\big|\sum_{k\in
\mathbb{Z}^d_*}
\widehat{\chi}_{\mathcal{B}_\theta}(tk)\widehat{\rho}(\varepsilon
k)\big| \,d\theta \lesssim 1, \label{Rd-f-ineq}
\end{equation}
where $d\theta$ is the normalized Haar measure on $SO(d)$, $\zeta_d$
and $\sigma(d, \omega)$ are given by \eqref{zeta_d} and
\eqref{sigma_d_w} respectively, and
\begin{equation*}
\varepsilon=\varepsilon(j, d, \omega)=2^{-j\alpha(d, \omega)},
\end{equation*}
\begin{equation*}
\alpha(d,\omega)=\left\{ \begin{array}{llll}
1-2[6(\omega-2)d^4+112(\omega-2)d^3-4(\omega-2)d^2  \\
\quad +(410-203\omega)d+82\omega-156]/\Box  & \textrm{for $3\leqslant d\leqslant 4$},\\
1-[4(\omega-2)d^4+90(\omega-2)d^3+61(\omega-2)d^2  \\
\quad -(227\omega-456)d+60\omega-112]/\triangle  & \textrm{for
$d\geqslant 5$},
\end{array}\right.
\end{equation*}
with $\Box$ and $\triangle$ given by \eqref{square} and
\eqref{triangle} respectively.
\end{lemma}

\begin{proof}
Let $t\in [2^{j-1}, 2^{j+2}]$ and $\delta=\delta(j, d,
\omega)=2^{-j\beta(d, \omega)}$ with
\begin{equation*}
\beta(d, \omega)=\left\{ \begin{array}{ll}
2(\omega-2)d(d-1)(d-2)/\Box & \textrm{for $3\leqslant d\leqslant 4$},\\
(\omega-2)d(d-1)(d-3)/\triangle & \textrm{for $d\geqslant 5$}.
\end{array}\right.
\end{equation*}

For any $\theta \in SO(d)$ we have the following splitting
\begin{equation*}
\sum_{k\in \mathbb{Z}^d_*}
\widehat{\chi}_{\mathcal{B}_\theta}(tk)\widehat{\rho}(\varepsilon
k)=\operatorname{Sum\ I}(t, \varepsilon, \delta,
\theta)+\operatorname{Sum\ II}(t, \varepsilon, \delta, \theta),
\end{equation*}
where
\begin{equation*}
\operatorname{Sum\ I}(t, \varepsilon, \delta, \theta)=\sum_{k\in
D_1(\delta, \theta)}
\widehat{\chi}_{\mathcal{B}_\theta}(tk)\widehat{\rho}(\varepsilon
k),
\end{equation*}
\begin{equation*}
\operatorname{Sum\ II}(t, \varepsilon, \delta, \theta)=\sum_{k\in
D_2(\delta, \theta)}
\widehat{\chi}_{\mathcal{B}_\theta}(tk)\widehat{\rho}(\varepsilon
k),
\end{equation*}
and $D_1(\delta, \theta)$, $D_2(\delta, \theta)$ are two regions
defined as follows:
\begin{equation*}
D_2(\delta, 0)=\{\xi\in \mathbb{R}^d_*: \xi/|\xi| \textrm{ or }
-\xi/|\xi| \in \vec{n}(\{x\in
\partial \mathcal{B}: K(x)<\delta \})\},
\end{equation*}
$D_1(\delta, 0)=\mathbb{R}^d_{*}\setminus D_2(\delta, 0)$,
$D_1(\delta, \theta)=\theta D_1(\delta, 0)$, and $D_2(\delta,
\theta)=\theta D_2(\delta, 0)$.

The estimate \eqref{Rd-f-ineq} follows from the next two claims.
Notice that the finite type condition is only used in the estimate
of Sum II and the size estimate of $|D_2(\delta, 0)|$.

\begin{claim}\label{sumII-finitetype}
\begin{equation*}
\int_{SO(d)} \sup_{2^{j-1}\leqslant t\leqslant 2^{j+2}}
t^{2d/(d+1)+\zeta_d+\sigma(d, \omega)}\big| \operatorname{Sum\
II}(t, \varepsilon, \delta, \theta)\big| \,d\theta \lesssim 1
\end{equation*}
with an implicit constant depending only on $\mathcal{B}$.
\end{claim}
\begin{claim} \label{sumI-finitetype}
\begin{equation*}
\int_{SO(d)} \sup_{2^{j-1}\leqslant t\leqslant 2^{j+2}}
t^{2d/(d+1)+\zeta_d+\sigma(d, \omega)}\big| \operatorname{Sum\ I}(t,
\varepsilon, \delta, \theta) \big| \,d\theta \lesssim 1
\end{equation*}
with an implicit constant depending only on $\mathcal{B}$.
\end{claim}
\end{proof}

\begin{proof}[Proof of Claim \ref{sumII-finitetype}]
\begin{equation*}
\textrm{L.H.S.}\lesssim (2^j)^{2d/(d+1)-(d+1)/2+\zeta_d+\sigma(d,
\omega)} \sum_{k\in \mathbb{Z}^d_*}
|k|^{-(d+1)/2}|\widehat{\rho}(\varepsilon k)|(\ast),
\end{equation*}
where
\begin{equation*}
(\ast):=\int_{SO(d)} 1_{D_2(\delta, \theta)}(k)
(\sup_{2^{j-1}\leqslant t\leqslant 2^{j+2}}
|tk|^{(d+1)/2}|\widehat{\chi}_{\mathcal{B}_\theta}(tk)|) \,d\theta.
\end{equation*}
Recalling the definition \eqref{def-Phi} of the function $\Phi$, we
get
\begin{align*}
(\ast)&\lesssim \int_{S^{d-1}\cap D_2(\delta, 0)} \Phi(\xi) \,d\xi,\\
&\lesssim \int_{S^{d-1}\cap D_2(\delta, 0)} (K_{\xi})^{-1/2}+(K_{-\xi})^{-1/2}+1 \,d\xi,\\
&\lesssim \int_{\vec{n}(\{x\in \partial \mathcal{B}: K(x)<
\delta \})} (K_{\xi})^{-1/2} \,d\xi,\\
&\lesssim \int_{\{x\in \partial \mathcal{B}: K(x)< \delta \}}
K(x)^{1/2} \,dA(x),\\
&\lesssim \delta^{1/2+1/(d-1)(\omega-2)}.
\end{align*}
In the above estimate of $(\ast)$ we have used Svensson's estimate
of $\Phi(\xi)$ for finite type domains (see \cite[p.~19]{svensson}),
the symmetry of $D_2(\delta, 0)$, a change of variables, and Lemma
\ref{s2-lemma3}. Hence
\begin{equation*}
\textrm{L.H.S.}\lesssim (2^j)^{2d/(d+1)-(d+1)/2+\zeta_d+\sigma(d,
\omega)}
\delta^{1/2+1/(d-1)(\omega-2)}\varepsilon^{-(d-1)/2}\lesssim 1.
\end{equation*}
\end{proof}

In order to prove Claim \ref{sumI-finitetype} we need an estimate of
$d$-dimensional exponential sums, which will be included in the
appendix for interested readers.

\begin{proof}[Proof of Claim \ref{sumI-finitetype}]
Note if $\xi \in D_1(\delta, \theta)$ then $K_{\pm
\xi}^{\theta}\geqslant \delta$. Applying Corollary \ref{s4-asym} to
Sum I yields
\begin{equation}
\operatorname{Sum\ I}(t, \varepsilon, \delta,
\theta)=(2\pi)^{-1}e^{\pi i(d+1)/4} S_1+(2\pi)^{-1}e^{-\pi i(d+1)/4}
\widetilde{S}_1+R_1,\label{decomp-sumI}
\end{equation}
where
\begin{equation*}
S_1(t, \varepsilon, \delta, \theta)=t^{-(d+1)/2}\sum_{k\in
D_1(\delta,
\theta)}|k|^{-(d+1)/2}(K_{k}^{\theta})^{-1/2}\widehat{\rho}(\varepsilon
k)e(tH_{\theta}(k)),
\end{equation*}
\begin{equation*}
\widetilde{S}_1(t, \varepsilon, \delta,
\theta)=t^{-(d+1)/2}\sum_{k\in D_1(\delta,
\theta)}|k|^{-(d+1)/2}(K_{-k}^{\theta})^{-1/2}\widehat{\rho}(\varepsilon
k)e(-tH_{\theta}(-k)),
\end{equation*}
and
\begin{equation}
R_1\lesssim \delta^{-2(d+1)}t^{-(d+3)/2}(\varepsilon^{-(d-3)/2}+\log
(\varepsilon^{-1}))\lesssim t^{-2d/(d+1)-\zeta_d-\sigma(d,
\omega)}.\label{R1}
\end{equation}

We will only estimate $S_1$ since $\widetilde{S}_1$ is similar.
Denote $\mathscr{C}_1=\{\xi \in \mathbb{R}^d: 1/2 \leqslant |\xi|
\leqslant 2 \}$. Let us introduce a dyadic decomposition and a
partition of unity.

Assume $\varphi \in C_0^{\infty}(\mathbb{R}^d)$ is a real radial
function such that $\supp(\varphi)\subset \mathscr{C}_1$,
$0\leqslant \varphi \leqslant 1$, and
\begin{equation*}
\sum_{l_0=-\infty}^{\infty}\varphi(\frac{\xi}{2^{l_0}})=1  \quad
\textrm{for }\xi\in \mathbb{R}^d\setminus \{0\}.
\end{equation*}
Denote
\begin{equation*}
S_{1, M}=\sum_{k\in D_1(\delta,
\theta)}\varphi(M^{-1}k)|k|^{-(d+1)/2}(K_{k}^{\theta})^{-1/2}\widehat{\rho}(\varepsilon
k)e(tH_{\theta}(k)),
\end{equation*}
then
\begin{equation}
S_1=t^{-(d+1)/2}\sum_{l_0=0}^{\infty}S_{1,
2^{l_0}}.\label{S1formula}
\end{equation}

We will estimate $S_{1, M}$ for a fixed $M=2^{l_0}$, $l_0\in
\mathbb{N}_0$. Let $q\in \mathbb{N}$. For each $\xi \in
\vec{n}((\partial \mathcal{B})_+)$ there exists a cone
\begin{equation*}
\mathfrak{C}(\xi, 2r(\xi)):=\mathop{\cup}\limits_{l>0} l B(\xi,
2r(\xi))\subset \mathbb{R}^d,
\end{equation*}
where $r(\xi)=c_2(K_\xi)^{d+2q+7-1/(d-1)}/2$ and $c_2$ is the
constant appearing in the statement of Lemma
\ref{non-vanishinglemma}. Note that Lemma \ref{curv-in-ball} implies
that $K_\eta \asymp K_\xi$ if $\eta\in \mathfrak{C}(\xi, 2r(\xi))$.
From the family of cones $\{\mathfrak{C}(\xi, r(\xi)/2): \xi \in
\vec{n}((\partial \mathcal{B})_+)\}$, we can choose, by a Vitali
procedure, a sequence $\{\mathfrak{C}(\xi_i,
r(\xi_i)/2)\}_{i=1}^{\infty}$ such that these cones still cover
$\vec{n}((\partial \mathcal{B})_+)$ and that $\{\mathfrak{C}(\xi_i,
r(\xi_i))\}_{i=1}^{\infty}$ satisfies the bounded overlap property.
Denote
\begin{equation*}
\mathfrak{C}_i^{\theta}=\theta\mathfrak{C}(\xi_i, r(\xi_i)).
\end{equation*}
Then the collection $\{\mathfrak{C}_i^{\theta}\}_{i=1}^{\infty}$
forms an open cover of $\vec{n}((\partial \mathcal{B}_{\theta})_+)$.
We can construct a partition of unity $\{\psi_i\}_{i=1}^{\infty}$
such that
\begin{enumerate}[\upshape (i)]
\item  $\sum_{i}\psi_i \equiv 1$ on $\vec{n}((\partial \mathcal{B}_{\theta})_+)$, and $\psi_i\in C_0^{\infty}(\mathfrak{C}_i^{\theta})$;
\item  each $\psi_i$ is positively homogeneous of degree zero;
\item  $|D^{\nu}\psi_i|\lesssim_{|\nu|} (K_{\xi_i})^{-(d+2q+7-1/(d-1))|\nu|}$ on $\mathscr{C}_1$.
\end{enumerate}

From the family $\{\mathfrak{C}_i^{\theta}\}_{i=1}^{\infty}$ we can
find a subfamily $\{\mathfrak{C}_i^{\theta}\}_{i\in \mathscr{A}}$
which covers $D_1(\delta, \theta)$, where
$\mathscr{A}=\mathscr{A}(\delta)$ is an index set such that $i\in
\mathscr{A}$ if and only if $\mathfrak{C}_i^{\theta}$ intersects
$D_1(\delta, \theta)$. Since $r(\xi_i)\gtrsim
\delta^{d+2q+7-1/(d-1)}$ for any $i\in \mathscr{A}$, a size estimate
gives that $\#\mathscr{A}\lesssim \delta^{-(d+2q+7-1/(d-1))(d-1)}$.
Define
\begin{equation}
S_{1, M}^*=\sum_{i\in \mathscr{A}}S_{2, i},\label{formula1}
\end{equation}
where
\begin{equation*}
S_{2, i}=\sum_{k\in \mathbb{Z}^d} U_i^{\theta}(k) e(tH_{\theta}(k))
\end{equation*}
and
\begin{equation*}
U_i^{\theta}(k)=\psi_i(M^{-1}k)\varphi(M^{-1}k)|k|^{-(d+1)/2}(K_{k}^{\theta})^{-1/2}\widehat{\rho}(\varepsilon
k).
\end{equation*}

Instead of $S_{1, M}$ we will estimate $S_{1, M}^*$. It turns out
that the error
\begin{equation}
R_{2, M}=S_{1, M}-S_{1, M}^* \label{R2M}
\end{equation}
is relatively small and this will be clear at the end of this proof.

To estimate $S_{1, M}^*$ we will estimate $S_{2, i}$ for any fixed
$i\in \mathscr{A}$. By Lemma \ref{non-vanishinglemma} and the
homogeneity of $H_{\theta}$, there exist $d$ linearly independent
vectors $v_j=v_j(\theta\xi_i)\in \mathbb{Z}^d$ ($j=1, \ldots, d$)
such that if $\eta \in \mathop{\cup}_{1/4\leqslant l\leqslant 4} l
B(\theta\xi_i, 2r(\xi_i))$ then
\begin{equation}
|h_q^{\theta}(\eta, v_1,\ldots, v_d)|\gtrsim
(K_{\xi_i})^{(-d-2q-5+1/(d-1))d(q+2)-3d+5-1/(d-1)}.
\label{determinant}
\end{equation}
Let $L=[\mathbb{Z}^d:\mathbb{Z}v_1\oplus\mathbb{Z}v_2\oplus
\cdots\oplus \mathbb{Z}v_d]$ be the index of the lattice spanned by
$v_1, \ldots, v_d$ in the lattice $\mathbb{Z}^d$. Then there exist
vectors $b_l\in \mathbb{Z}^d$ ($l=1, \ldots, L$) such that
\begin{equation*}
\mathbb{Z}^d=\biguplus_{l=1}^{L}(\mathbb{Z}v_1+\cdots+\mathbb{Z}v_d+b_l).
\end{equation*}
It follows from Lemma \ref{non-vanishinglemma} that
\begin{equation*}
L=|\det(v_1, \ldots, v_d)| \asymp (K_{\xi_i})^{d(-d-2q-5+1/(d-1))}
\end{equation*}
and
\begin{equation*}
|b_l|\lesssim (K_{\xi_i})^{-d-2q-8+1/(d-1)}.
\end{equation*}

Let $N>d/2$ be an arbitrarily fixed natural number. We have
\begin{align}
S_{2, i}&=\sum_{l=1}^{L}\sum_{m\in \mathbb{Z}^d}U_i^{\theta}(\sum_{j=1}^{d} m_jv_j+b_l) e(tH_{\theta}(\sum_{j=1}^{d} m_jv_j+b_l))\nonumber\\
        &=(K_{\xi_i})^{-1/2}M^{-(d+1)/2}(1+M\varepsilon)^{-N}\sum_{l=1}^{L}S(T, M_*; G_l, F_l),\label{formula2}
\end{align}
where $T=tM$, $M_*=(K_{\xi_i})^{d+2q+2-1/(d-1)}M$,
\begin{equation*}
G_l(y)=(K_{\xi_i})^{1/2}M^{(d+1)/2}(1+M\varepsilon)^{N}
U_i^{\theta}(M_*Vy+b_l),
\end{equation*}
and
\begin{equation*}
F_l(y)=H_{\theta}(M^{-1}(M_*Vy+b_l)),
\end{equation*}
where $V=(v_1,\ldots, v_d)$.

We consider the function $F_l$ restricted to the convex domain
\begin{equation*}
\Omega_l=\{y\in \mathbb{R}^d : M^{-1}(M_*Vy+b_l)\in
\mathop{\cup}\limits_{1/4\leqslant l\leqslant 4} l B(\theta\xi_i,
2r(\xi_i)) \}.
\end{equation*}
The support of $G_l$ satisfies
\begin{equation*}
\supp(G_l)\subset \{y\in \mathbb{R}^d : M^{-1}(M_*Vy+b_l)\in
\overline{\mathscr{C}_1\cap \mathfrak{C}_i^{\theta}} \}\subset
\Omega_l.
\end{equation*}

We apply to $S(T, M_*; G_l, F_l)$ Proposition \ref{exp-sum} with
$G=G_l$, $F=F_l$, $K=K_{\xi_i}$, and $\Omega=\Omega_l$. And we only
compute below the case $d\geqslant 5$ with $q=1$ (while the case
$3\leqslant d\leqslant 4$ with $q=2$ can be handled in the same
way).

Since $1 \gtrsim K_{\xi_i}\gtrsim \delta$ if $i\in \mathscr{A}$,
there exist positive constants $C_2$ and $C_3$ such that the
assumptions of Proposition \ref{exp-sum} are satisfied if $M\in
\mathscr{I}_1$ where $\mathscr{I}_1$ is an interval defined by
\begin{equation*}
\mathscr{I}_1=[C_3\delta^{-37d-41+5/(d-1)},
C_2\delta^{(14d^4+66d^3+61d^2-144d-12)/(2(d-2)(d-1)d)} t^{d/(d-2)}].
\end{equation*}
This follows from Lemma \ref{non-vanishinglemma},
\eqref{determinant} and the following facts: if
\begin{equation*}
(K_{\xi_i})^{-d-2q-8+1/(d-1)}\lesssim M
\end{equation*}
then $\Omega_l\subset c_0 B(0, 1)$ for a constant $c_0$ (depending
only on $q$, $\mathcal{B}$);
\begin{equation*}
\dist \left(\big(\mathop{\cup}\limits_{1/4\leqslant l\leqslant 4} l
B(\theta\xi_i, 2r(\xi_i)) \big)^c \, , \,
\overline{\mathscr{C}_1\cap \mathfrak{C}_i^{\theta}}\right)\geqslant
c_2 (K_{\xi_i})^{d+2q+7-1/(d-1)}/8;
\end{equation*}
and
\begin{equation*}
D^{\nu}U_i^{\theta} \lesssim
(K_{\xi_i})^{-(d+2q+7-1/(d-1))|\nu|-1/2}M^{-|\nu|-(d+1)/2}(1+M\varepsilon)^{-N}.
\end{equation*}

Thus by Proposition \ref{exp-sum} we get
\begin{equation}
S(T, M_*; G_l, F_l)\lesssim
(K_{\xi_i})^{d^2-13d/2-6+9/(d+2)}t^{d/(2d+4)}M^{d-d/(d+2)}.\label{kk}
\end{equation}
Then by using \eqref{formula1}, \eqref{formula2}, \eqref{kk},
$K_{\xi_i}\gtrsim \delta$, and bounds of $\#\mathscr{A}$ and $L$, we
get
\begin{equation}
\begin{split}
S_{1, M}^* &\lesssim \delta^{-d^2-43d/2+9/2+1/(d-1)+9/(d+2)}\cdot \\
&\qquad t^{d/(2d+4)}M^{(d-1)/2-d/(d+2)}(1+M\varepsilon)^{-N}.
\end{split}\label{S1M}
\end{equation}

Now we can estimate $S_1$. By \eqref{S1formula} and \eqref{R2M} we
get
\begin{equation}
S_1=t^{-(d+1)/2}\left( \sum_{l_0\in \{n\in \mathbb{N}_0 : 2^{n}\in
\mathscr{I}_1\}} S_{1, 2^{l_0}}^*+R_2+R_3\right),\label{boundS1}
\end{equation}
where
\begin{equation*}
R_2=\sum_{l_0\in \{n\in \mathbb{N}_0 : 2^{n}\in \mathscr{I}_1\}}
R_{2, 2^{l_0}}\quad \textrm{and} \quad R_3=\sum_{l_0\in \{n\in
\mathbb{N}_0 : 2^{n}\in \mathscr{I}_1^{c}\}}S_{1, 2^{l_0}}.
\end{equation*}

Using the bound \eqref{S1M} of $S_{1, M}^*$ we get
\begin{equation}
\begin{split}
&\sum_{l_0\in \{n\in \mathbb{N}_0 : 2^{n}\in \mathscr{I}_1\}} S_{1,
2^{l_0}}^*\lesssim \\
&\qquad\delta^{-d^2-43d/2+9/2+1/(d-1)+9/(d+2)}t^{d/(2d+4)}\varepsilon^{-(d-1)/2+d/(d+2)}.
\end{split}\label{qq}
\end{equation}

Hence Claim \ref{sumI-finitetype} follows from \eqref{decomp-sumI},
\eqref{R1}, \eqref{boundS1}, \eqref{qq}, sizes of $\delta$ and
$\varepsilon$, and the following estimates of $R_2$ (\footnote{The
method used here to estimate $R_2$ is different from what we used in
\cite{guo2}. More precisely, we estimate the integral of $R_2$
rather than $R_2$ itself.  Here we need the size estimate of
$|D_2(\delta, 0)|$, and this is the only place in the estimate of
Sum I where the finite type condition is used.}) and $R_3$:
\begin{equation}
\int_{SO(d)} \sup_{2^{j-1}\leqslant t\leqslant 2^{j+2}}
t^{2d/(d+1)+\zeta_d+\sigma(d, \omega)}t^{-(d+1)/2}|R_2| \,d\theta
\lesssim 1;\label{R2-remainder}
\end{equation}
\begin{equation} \sup_{2^{j-1}\leqslant t\leqslant 2^{j+2}}
t^{2d/(d+1)+\zeta_d+\sigma(d, \omega)}t^{-(d+1)/2}|R_3| \lesssim
1.\label{R3-remainder}
\end{equation}

The \eqref{R2-remainder} follows from Lemma \ref{s2-lemma2} and
\ref{s2-lemma3} if we notice that
\begin{equation*}
|R_{2, M}|\lesssim \delta^{-1/2}\sum_{k\in D_2(\delta,
\theta)}|\varphi(M^{-1}k)| |k|^{-(d+1)/2}|\widehat{\rho}(\varepsilon
k)|;
\end{equation*}
and the \eqref{R3-remainder} is true since we have, by trivial
estimates,
\begin{equation*}
|R_3|\lesssim \delta^{-5d^2-7d/2+9}+
\delta^{-1/2}t^{d(d-1)/(2(d-2))-N_1d/(d-2)}\varepsilon^{-N_1},
\end{equation*}
for any integer $N_1>(d-1)/2$.

\end{proof}

Just like that Lemma \ref{lemma5:1} implies Theorem
\ref{Rd-finitetype}, the following lemma implies Theorem
\ref{Rd-varchenkoplus}.  Its proof is essentially the same as above,
however, we now use, in the estimate of Sum II, H\"{o}lder's
inequality and Varchenko's \cite[Theorem 8]{varchenko} instead of
Svensson's estimate of $\Phi(\xi)$ (as we mentioned in \S
\ref{introduction}) and $\delta=t^{-\beta(d, \infty)}$ with
$\beta(d, \infty)=\lim_{\omega\rightarrow \infty}\beta(d,\omega)$.

\begin{lemma}
Let $\mathcal{B}\subset \mathbb{R}^d$ ($d\geqslant 3$) be a compact
convex domain and $\rho\in C_0^{\infty}(\mathbb{R}^d)$ such that
$\int_{\mathbb{R}^d}\rho(y)\,dy=1$. If the boundary is a smooth
hypersurface then
\begin{equation*}
\int_{SO(d)} t^{2d/(d+1)+\zeta_d}\big|\sum_{k\in \mathbb{Z}^d_*}
\widehat{\chi}_{\mathcal{B}_\theta}(tk)\widehat{\rho}(\varepsilon
k)\big| \,d\theta \lesssim 1,
\end{equation*}
where $d\theta$ is the normalized Haar measure on $SO(d)$, $\zeta_d$
is given by \eqref{zeta_d}, and
\begin{equation*}
\varepsilon=t^{-\alpha(d, \infty)}
\end{equation*}
with
\begin{equation*}
\alpha(d,\infty)=\left\{ \begin{array}{llll}
1-\frac{12d^4+224d^3-8d^2-406d+164}{6d^5+118d^4+109d^3-210d^2-119d+82}  & \textrm{for $3\leqslant d\leqslant 4$},\\
1-\frac{4d^4+90d^3+61d^2-227d+60}{2d^5+47d^4+76d^3-85d^2-82d+30}  &
\textrm{for $d\geqslant 5$}.
\end{array}\right.
\end{equation*}
\end{lemma}

\begin{remark}
Note that $\alpha(d,\infty)=\lim_{\omega\rightarrow
\infty}\alpha(d,\omega)$.
\end{remark}


\section{The $\mathbb{R}^2$ case}\label{R2-section}

To prove Theorem \ref{R2-general} and \ref{R2-finitetype} the key
step is to prove the following $\mathbb{R}^2$ analogues of Lemma
\ref{lemma5:1}.

\begin{lemma}\label{lemma6:1}
Let $\zeta_2=1/2859$, $\mathcal{B}\subset \mathbb{R}^2$ be a compact
convex domain with a smooth boundary, and $\rho\in
C_0^{\infty}(\mathbb{R}^2)$ such that
$\int_{\mathbb{R}^2}\rho(y)\,dy=1$. Then, for $j\in \mathbb{N}$, we
have
\begin{equation*}
\int_{SO(2)} \sup_{2^{j-1}\leqslant t\leqslant 2^{j+2}}
t^{4/3+\zeta_2}\big|\sum_{k\in \mathbb{Z}^2_*}
\widehat{\chi}_{\mathcal{B}_\theta}(tk)\widehat{\rho}(\varepsilon(j,
\infty) k)\big| \,d\theta \lesssim 1,
\end{equation*}
where $d\theta$ is the normalized Haar measure on $SO(2)$ and
$\varepsilon(j, \infty)=2^{-318j/953}$. Furthermore, if the boundary
is of finite type $\omega$ then
\begin{equation*}
\int_{SO(2)} \sup_{2^{j-1}\leqslant t\leqslant 2^{j+2}}
t^{4/3+\zeta_2+\sigma(2, \omega)}\big|\sum_{k\in \mathbb{Z}^2_*}
\widehat{\chi}_{\mathcal{B}_\theta}(tk)\widehat{\rho}(\varepsilon(j,
\omega) k)\big| \,d\theta \lesssim 1,
\end{equation*}
where $\varepsilon(j, \omega)=2^{-j\alpha(2, \omega)}$,
\begin{equation*}
\alpha(2, \omega)=\frac{318\omega-616}{953\omega-1848} \quad
\textrm{and} \quad \sigma(2,
\omega)=\frac{616}{953(953\omega-1848)}.
\end{equation*}
\end{lemma}

Since the proof is essentially the same as the proof of Lemma
\ref{lemma5:1} we will not provide every detail but only a few key
estimates (see also the proof of Guo's \cite[Lemma 6.1]{guo2}).

As before we first decompose $\sum_{k\in \mathbb{Z}^2_*}
\widehat{\chi}_{\mathcal{B}_\theta}(tk)\widehat{\rho}(\varepsilon
k)$ into two parts: Sum I and II.

By Lemma \ref{s2-lemma2} and the fact that $\Phi\in L^{2,
\infty}(S^1)$ (\cite[Theorem 0.3]{brandolini}) we get
\begin{equation*}
\int_{SO(2)} \sup_{2^{j-1}\leqslant t\leqslant 2^{j+2}}
t^{4/3+\zeta_2} \big| \operatorname{Sum\ II}\big| \,d\theta \lesssim
(2^{j})^{4/3-3/2+\zeta_2} \delta^{1/2}\varepsilon^{-1/2}\lesssim 1,
\end{equation*}
where $\varepsilon=\varepsilon(j, \infty)$ and $\delta=\delta(j,
\infty)=2^{-j/953}$.

If $\partial \mathcal{B}$ is of finite type $\omega$, then
\begin{align*}
\int_{SO(2)} &\sup_{2^{j-1}\leqslant t\leqslant 2^{j+2}}
t^{4/3+\zeta_2+\sigma(2, \omega)}\big| \operatorname{Sum\ II}\big|
\,d\theta \\
&\lesssim (2^{j})^{4/3-3/2+\zeta_2+\sigma(2, \omega)}
\delta^{1/2+1/(\omega-2)}\varepsilon^{-1/2}\ \lesssim 1,
\end{align*}
where $\varepsilon=\varepsilon(j, \omega)$ and $\delta=\delta(j,
\omega)=2^{-j\beta(2, \omega)}$ with
\begin{equation*}
\beta(2, \omega)=\frac{\omega-2}{953\omega-1848}.
\end{equation*}

For Sum I we now use Lemma \ref{r2-keylemma} and get
\begin{equation*}
\operatorname{Sum\ I}\lesssim
\delta^{-14}t^{-3/2+1/22}\varepsilon^{-7/22}+t^{-3/2}|R_2|.
\end{equation*}
Combining this with the above two estimates of Sum II yields Lemma
\ref{lemma6:1}.



\appendix

\section{Several Lemmas}\label{app1}

Here is a quantitative version of the inverse function theorem (see
the appendix in Guo~\cite{guo}).

\begin{lemma}\label{app:lemma:1}
Suppose $f$ is a $C^{(k)}$ ($k\geqslant 2$) mapping from an open set
$\Omega\subset\mathbb{R}^d$ into $\mathbb{R}^d$ and $b=f(a)$ for
some $a\in \Omega$. Assume $|\det (\nabla f(a))|$ $\geqslant$ $c$
and for any $x\in\Omega$,
\begin{equation*}
|D^{\nu} f_i(x)|\leqslant C \quad \quad \textrm{for $|\nu|\leqslant
2$, $1\leqslant i\leqslant d$}.
\end{equation*}
If $r_0\leqslant \sup\{r>0: B(a, r)\subset \Omega\}$, then $f$ is
bijective from $B(a, r_1)$ to an open set containing $B(b, r_2)$
where
\begin{equation*}
r_1=\min\{\frac{c}{2d^{2} d! C^d}, r_0\},  \qquad
r_2=\frac{c}{4d!C^{d-1}}r_1.
\end{equation*}
The inverse mapping $f^{-1}$ is in $C^{(k)}(V)$.
\end{lemma}

H\"ormander's \cite[Theorem 7.7.1]{hormander}  gives the following
estimate obtained by integration by parts.

\begin{lemma}\label{app:lemma:2}
Let $K\subset \mathbb{R}^d$ be a compact set, $X$ an open
neighborhood of $K$ and $k$ a nonnegative integer. If $u\in
C_0^k(K)$, real $f\in C^{k+1}(X)$, then
\begin{equation*}
|\int u(x) e^{i \lambda f(x)}\,dx|\leqslant C |K|\lambda^{-k}
\sum_{|\nu|\leqslant k} \sup |D^{\nu}u||\nabla f|^{|\nu|-2k}, \quad
\lambda>0.
\end{equation*}
Here $C$ is bounded when $f$ stays in a bounded set in $C^{k+1}(X)$.
\end{lemma}

The following lemmas are various results of the method of stationary
phase. The first one is H\"ormander's \cite[Lemma 7.7.3]{hormander}.
The second one is Sogge and Stein's \cite[Lemma 2]{soggestein}.

\begin{lemma}\label{app:lemma:hor}
Let $A$ be a symmetric non-degenerate matrix with $Im A\geqslant 0$.
Then we have for every integer $k>0$ and integer $s>d/2$
\begin{equation*}
\begin{split}
|\int u(x) e^{i \lambda\langle Ax, x\rangle/2} &\, dx-(\det(\lambda
A/2\pi i))^{-1/2}T_k(\lambda)|\\
&\leqslant C_k(\|A^{-1}\|/\lambda)^{d/2+k}\sum_{|\alpha|\leqslant
2k+s} \|D^{\alpha}u\|_{L^2}, \quad u\in \mathscr{S},
\end{split}
\end{equation*}
\begin{equation*}
T_k(\lambda)=\sum_{0}^{k-1}(2i\lambda)^{-j}\langle A^{-1}D,
D\rangle^j u(0)/j!.
\end{equation*}
\end{lemma}

\begin{lemma}\label{app:lemma:3}
Suppose $\phi$ and $\psi$ are smooth functions in $B(0,
\delta)\subset \mathbb{R}^d$ with $\phi$ real-valued. Assume that
$|(\partial/\partial x)^{\nu}\phi|\leqslant C_1$, $|\nu|\leqslant
d+2$ and $|(\partial/\partial x)^{\nu}\psi|\leqslant C_2
\delta^{-|\nu|}$, $|\nu|\leqslant d$. We also suppose that $(\nabla
\phi)(0)=0$, but $|\det \nabla^2\phi(0)|\geqslant \delta$. Then
there exists a positive constant $c_1$ (independent of $\delta$),
which is sufficiently small, so that if $\psi$ is supported in $B(0,
c_1\delta)$ we can assert that
\begin{equation*}
\big|\int_{\mathbb{R}^d} \psi e^{i\lambda \phi}\,dx \big|\leqslant
C\lambda^{-d/2}\delta^{-1/2}.
\end{equation*}
\end{lemma}

\section{Estimate of Exponential Sums}\label{app2}

In this section we will prove a higher dimensional analogue of Guo's
\cite[Proposition 5.2]{guo2} by using the same method.

Let $M_*>1$ and $T>0$ be parameters. We consider $d$-dimensional
exponential sums of the form
\begin{equation*}
S(T, M_*; G, F)=\sum_{m\in\mathbb{Z}^d} G(m/M_*) e(TF(m/M_*)),
\end{equation*}
where $G:\mathbb{R}^d\rightarrow\mathbb{R}$ is smooth, compactly
supported, and bounded above by a constant, and $F:\Omega\subset
\mathbb{R}^d \rightarrow\mathbb{R}$ is smooth on an open convex
domain $\Omega$ such that
\begin{equation*}
\textrm{supp} (G)\subset \Omega \subset c_0 B(0, 1),
\end{equation*}
where $c_0>0$ is a fixed constant.

\begin{proposition}\label{exp-sum}
Let $d\geqslant 3$, $q\in \mathbb{N}$, $Q=2^q$, and $K<1$ be a
positive parameter. Assume that
\begin{equation}
\dist(\supp(G), \Omega^{c})\gtrsim K^{d+2q+13-1/(d-1)},
\label{bdy-dist}
\end{equation}
for all $\nu\in \mathbb{N}_0^d$ and $y\in \Omega$
\begin{equation}
D^{\nu}G(y)\lesssim K^{-(d+2q+13-1/(d-1))|\nu|}, \label{aaa}
\end{equation}
\begin{equation}
D^{\nu}F(y)\lesssim \bigg\{ \begin{array}{ll}
K^{-6|\nu|} & \textrm{if $0\leqslant |\nu|\leqslant 1$},\\
K^{3-8|\nu|} & \textrm{if $|\nu|\geqslant 2$},
\end{array}  \label{upperforf}
\end{equation}
and for $\mu=(1, 0, \ldots, 0, q-1)\in \mathbb{N}_0^d$
\begin{equation}
\big|\det( \nabla^2 D^{\mu}F(y) )\big|\gtrsim
K^{-3(q+3)d+5-1/(d-1)}. \label{lowerforf}
\end{equation}

If
\begin{equation}
M_*\geqslant K^{-4(5q+4)d-37+4/(d-1)} \label{res-M}
\end{equation}
and
\begin{equation}
T\geqslant K^{-I/(Q(d-1)d^2)}M_*^{q+2/Q-2/d} \label{restriction}
\end{equation}
with
\begin{equation*}
\begin{split}
I&=2(5q+4)(2Q-3)d^4+(-35+25q+40Q+2qQ)d^3+\\
&\quad (60+5q-Q-17qQ)d^2+(6-60Q-5qQ)d-6Q,
\end{split}
\end{equation*}
then
\begin{equation}
\begin{split}
S(T, M_*; G, F)\lesssim
&[K^{-(2(5q+4)d^3+(3q+19)d^2-(13q+24)d-6)/(d(d-1))}\cdot \\
&\quad TM_*^{2(Q-1)d+2Q-q-2}]^{d/(2Q+2(Q-1)d)}. \label{kkk}
\end{split}
\end{equation}

The constant implicit in \eqref{kkk} depends only on $d$, $q$,
$c_0$, and constants implicit in \eqref{bdy-dist}, \eqref{aaa},
\eqref{upperforf}, and \eqref{lowerforf}.
\end{proposition}

\begin{proof} Let $H$ be a parameter satisfying
\begin{equation}
1<H\leqslant c_5 K^{(6(5q+4)d^3-5(5q-7)d^2-5(q+12)d-6)/(2(d-1)d)}M_*
\label{H}
\end{equation}
with $c_5<1$ chosen (later) to be sufficiently small. Then
$H\leqslant M_*$.  We apply to $S(T, M_*; G, F)$ the iterated
one-dimensional Weyl-Van der Corput inequality
 with $r_1=e_1$ and $r_j=e_d$ ($j=2, \ldots,
q$) (see \cite[Lemma 2.2]{guo} for this inequality and notations
like $G_q$, $F_q$, $\mathscr{H}$, and $\Omega_q$ that we will use
below). Then we need to estimate $S_4:=S(\mathscr{H}TM_*^{-q}, M_*;
G_q, F_q)$. Applying the Poisson summation formula followed by a
change of variables yields
\begin{equation*}
S_4=\sum_{p\in \mathbb{Z}^d} K^{8d}M_*^d\int_{\mathbb{R}^d}
\Psi_q(z)e(\mathscr{H}TM_*^{-q}F_q(K^8 z)-K^{8}M_* \langle p,
z\rangle)\,dz,
\end{equation*}
where $\Psi_q(z)=G_q(K^{8}z)$. It is obvious that
\begin{equation}
\supp(\Psi_q)\subset K^{-8}\Omega_q \subset c_0 K^{-8}B(0, 1).
\label{domain1}
\end{equation}
By \eqref{bdy-dist} we also have
\begin{equation}
\dist(\supp(\Psi_q), (K^{-8}\Omega_q)^{c})\gtrsim
K^{d+2q+5-1/(d-1)}. \label{domain2}
\end{equation}

By the assumption \eqref{upperforf} there exists a constant $A_1$
such that
\begin{equation*}
|\nabla_z(F_q(K^{8}z))|\leqslant (A_1/2) K^{3-8q}.
\end{equation*}
We split $S_4$ into two parts
\begin{equation*}
S_4=\sum_{|p|<A_1K^{-8q-5}\mathscr{H}TM_*^{-q-1}}+\sum_{|p|\geqslant
A_1K^{-8q-5}\mathscr{H}TM_*^{-q-1}}=\textrm{:}S_5+R_5.
\end{equation*}

It is not hard to prove, by integration by parts (Lemma
\ref{app:lemma:2}), that
\begin{equation}
R_5\lesssim K^{-(d+2q+13-1/(d-1))(d+1)}M_*^{-1}.\label{R5}
\end{equation}

Define $\lambda_1=K^{3-8q}\mathscr{H}TM_*^{-q}$ and
\begin{equation*}
\Phi_q(z, p)=(\mathscr{H}TM_*^{-q}F_q(K^8 z)-K^{8}M_* \langle p,
z\rangle)/\lambda_1,
\end{equation*}
then
\begin{equation}
S_5=K^{8d}M_*^d \sum_{|p|<A_1K^{-8}\lambda_1M_*^{-1}}
\int_{\mathbb{R}^d} \Psi_q(z)e(\lambda_1 \Phi_q(z, p))\,dz.
\label{S5}
\end{equation}

To estimate $S_5$ we discuss in two cases.

CASE 1: $\lambda_1\geqslant K^{-4(5q+4)d-29+4/(d-1)}$.

For all $z\in K^{-8}\Omega_q$, by \eqref{aaa}, \eqref{upperforf},
and \eqref{lowerforf}, we get
\begin{equation}
D^{\nu}_z\Psi_q(z)\lesssim K^{-(d+2q+5-1/(d-1))|\nu|}, \label{ff}
\end{equation}
\begin{equation}
D^{\nu}_z \Phi_q(z, p)\lesssim \bigg\{
\begin{array}{ll}
K^{-8} & \textrm{for $\nu=0$},\\
1 & \textrm{for $|\nu|\geqslant 1$},
\end{array}   \label{gg}
\end{equation}
and
\begin{equation}
|\det\big(\nabla^2_{zz}\Phi_q(z, p) \big)|\gtrsim
K^{(5q+4)d+5-1/(d-1)}. \label{cc}
\end{equation}

To prove this lower bound \eqref{cc} we first note, by using the
definition of $F_q$ and the mean value theorem, that for $\mu=(1, 0,
\ldots, 0, q-1)\in \mathbb{N}_0^d$
\begin{equation*}
\frac{\partial^2}{\partial z_{l_1} \partial z_{l_2}} (\Phi_q(z,
p))=K^{8q+13}\frac{\partial^{2} D^{\mu}F}{\partial x_{l_1} \partial
x_{l_2}}(K^{8}z)+O(K^{-8}\frac{H}{M_*}).
\end{equation*}
The two terms on the right are $\lesssim$ $1$ and $c_5
K^{(5q+4)d+5-1/(d-1)}$ respectively (implied by \eqref{upperforf}
and \eqref{H}). Thus
\begin{equation*}
\det(\nabla^2_{zz}(\Phi_q(z,
p)))=K^{(8q+13)d}\det(\nabla^2D^{\mu}F)+O(c_5
K^{(5q+4)d+5-1/(d-1)}).
\end{equation*}
By \eqref{lowerforf}, we get \eqref{cc} if we pick a sufficiently
small $c_5$.

With \eqref{domain1}, \eqref{domain2}, \eqref{ff}, \eqref{gg}, and
\eqref{cc}, we can estimate the integrals in $S_5$. Let us fix an
arbitrary $p\in\mathbb{Z}^d$ with $|p|<A_1K^{-8}\lambda_1M_*^{-1}$.

We first need to estimate the number of critical points of the phase
function $\Phi_q(z, p)$. Denote $\widetilde{p}=K^{8}M_*p/\lambda_1$
and $F(z)=K^{8q-3}\nabla_z(F_q(K^{8}z))$, then $\nabla_z \Phi_q(z,
p)=F(z)-\widetilde{p}$ and the critical points are determined by the
equation
\begin{equation*}
F(z)=\widetilde{p} \quad \textrm{for $z\in K^{-8}\Omega_q$}.
\end{equation*}
The bounds \eqref{gg} and \eqref{cc} imply that the mapping $F$ and
its components $F_j$ satisfy
\begin{equation*}
D^{\nu}F_j(z)\lesssim 1   \quad \textrm{for $|\nu|\leqslant 2$,
$j=1$, \ldots, $d$},
\end{equation*}
and
\begin{equation*}
|\det(\nabla_z F(z))|\gtrsim K^{(5q+4)d+5-1/(d-1)}.
\end{equation*}
By \eqref{domain2}, we know that $\supp(\Psi_q)$ is strictly smaller
than the domain $K^{-8}\Omega_q$ and the distance between their
boundary is larger than $a_1 K^{d+2q+5-1/(d-1)}$ for some positive
constant $a_1$. Let $r_0=a_1 K^{d+2q+5-1/(d-1)}/2$. By Taylor's
formula, there exists a positive constant $a_2$ ($<a_1/2$) such that
if $\tilde{z}$ is a critical point in $(\supp(\Psi_q))_{(r_0)}$
then, for any $z\in B(\tilde{z}, a_2 K^{d+2q+5-1/(d-1)})$,
\begin{equation}
|\nabla_z \Phi_q(z, p)|\gtrsim K^{(5q+4)d+5-1/(d-1)}|z-\tilde{z}|.
\label{taylor}
\end{equation}
Applying Lemma \ref{app:lemma:1} to $F$ with $r_0$ as above yields
two positive constants $a_3$ ($<a_2/2$) and $a_4$ such that if
\begin{equation*}
r_1=a_3 K^{(5q+4)d+5-1/(d-1)}\quad \textrm{and}\quad r_2=a_4
K^{2((5q+4)d+5-1/(d-1))},
\end{equation*}
then $F$ is bijective from $B(z, 2r_1)$ to an open set containing
$B(F(z), 2r_2)$ for any $z\in (\supp(\Psi_q))_{(r_0)}$. It follows,
simply by a size estimate, that the number of critical points in
$(\supp(\Psi_q))_{(r_1)}$ is $\lesssim$ $(K^{-8}/r_1)^{d}\lesssim
K^{-((5q+4)d+13-1/(d-1))d}$.

For the $p$ that we have fixed, let $Z_j$ ($j=1, \ldots, J(p)$) be
all critical points in $(\supp(\Psi_q))_{(r_1)}$ of the phase
function $\Phi_q(z, p)$ and $\chi_j(z)=\chi_0((z-Z_j)/(c_6 r_1))$
with $c_6$ chosen below. Then the integral in $S_5$ can be
decomposed as
\begin{equation}
\int \Psi_q(z)e(\lambda_1 \Phi_q(z, p))\,dz =S_6+R_6,\label{S5-1}
\end{equation}
where
\begin{equation*}
S_6=\sum_{j=1}^{J(p)}\int \chi_j(z) \Psi_q(z)e(\lambda_1 \Phi_q(z,
p))\,dz
\end{equation*}
and
\begin{equation*}
R_6=\int \big(1-\sum_{j=1}^{J(p)}\chi_j(z)\big) \Psi_q(z)e(\lambda_1
\Phi_q(z, p))\,dz.
\end{equation*}

It follows from integration by parts (Lemma \ref{app:lemma:2}) and
\eqref{taylor} that
\begin{equation}
R_6\lesssim K^{-8d-4((5q+4)d+7-1/(d-1))N}\lambda_1^{-N}.
\label{boundR6}
\end{equation}

As to $S_6$, for each $1\leqslant j\leqslant J(p)$, let $\phi_j(z,
p)=\Phi_q(z, p)-\Phi_q(Z_j, p)$. By Lemma \ref{app:lemma:3}, if
$c_6$ is sufficiently small then
\begin{align}
\big|\int \chi_j(z) \Psi_q(z)e(\lambda_1 \Phi_q(z, p))\,dz\big|&= \nonumber \\
  \big|\int \chi_j(z) \Psi_q(z)e(\lambda_1 \phi_j(z, p))\,dz\big|&\lesssim K^{-((5q+4)d+5-1/(d-1))/2}\lambda_1^{-d/2}.\label{S5-2}
\end{align}
Hence
\begin{equation}
S_6\lesssim K^{-8d-((5q+4)d+5-1/(d-1))(d+1/2)}\lambda_1^{-d/2}.
\label{S5int}
\end{equation}

Noticing that we have assumed $\lambda_1\geqslant
K^{-4(5q+4)d-29+4/(d-1)}$ in the Case 1, it is then easy to check
that the bound \eqref{boundR6} of $R_6$ is less than the bound
\eqref{S5int} of $S_6$ if $N$ is sufficiently large. Hence, by
\eqref{S5}, \eqref{S5-1}, \eqref{S5int}, and \eqref{boundR6}, we get
the following bound of $S_5$
\begin{equation*}
S_5\lesssim K^{8d}M_*^d((K^{-8}\lambda_1M_*^{-1})^d+1)
K^{-8d-((5q+4)d+5-1/(d-1))(d+1/2)}\lambda_1^{-d/2},
\end{equation*}
\begin{equation}
\lesssim
K^{-((5q+4)d+5-1/(d-1))(d+1/2)}(K^{-8d}\lambda_1^{d/2}+M_*^d\lambda_1^{-d/2}).\label{case1}
\end{equation}

CASE 2: $\lambda_1< K^{-4(5q+4)d-29+4/(d-1)}$.

Within this range of $\lambda_1$, the assumption \eqref{res-M}
implies $K^{-8}\lambda_1M_*^{-1}<1$, hence the trivial estimate of
$S_5$ (together with \eqref{domain1} and \eqref{ff}) yields
\begin{equation}
S_5\lesssim M_*^d\leqslant K^{-(4(5q+4)d+29-4/(d-1))d/2}M_*^d
\lambda_1^{-d/2}. \label{case2}
\end{equation}

Combining the bounds of $S_5$ from Case 1 and 2 (namely,
\eqref{case1} and \eqref{case2}) yields
\begin{equation*}
\begin{split}
S_5&\lesssim K^{-8d-((5q+4)d+5-1/(d-1))(d+1/2)}\lambda_1^{d/2}\\
   &\quad +K^{-(4(5q+4)d+29-4/(d-1))d/2}M_*^d\lambda_1^{-d/2}.
\end{split}
\end{equation*}

Note that this bound of $S_5$ is larger than the bound \eqref{R5} of
$R_5$ no matter whether $\lambda_1\leqslant 1$ or $\lambda_1>1$. It
follows that
\begin{equation*}
\begin{split}
S_4=S_5+R_5 &\lesssim K^{-(4q+13/2)d-((5q+4)d+5-1/(d-1))(d+1/2)}(\mathscr{H}TM_*^{-q})^{d/2}\\
            &\quad
            +K^{-2d((5q+4)d+8-2q-1/(d-1))}(\mathscr{H}TM_*^{-q-2})^{-d/2},
\end{split}
\end{equation*}
where we have already used the definition of $\lambda_1$.

Plugging this bound of $S_4$ into the Weyl-Van der Corput inequality
that we used at the beginning gives
\begin{equation}
\begin{split}
|S(T, M_*; G, F)|^{Q}&\lesssim
M_*^{dQ}H^{-1}+H^{(1-1/Q)d}T^{d/2}M_*^{d(Q-1-q/2)}\cdot\\
 &\quad K^{-(4q+13/2)d-((5q+4)d+5-1/(d-1))(d+1/2)}+\textrm{E},
\label{cccc}
\end{split}
\end{equation}
where
\begin{equation*}
\textrm{E}=K^{-2d((5q+4)d+8-2q-1/(d-1))}H^{-2+2/Q}T^{-d/2}M_*^{d(Q+q/2)}.
\end{equation*}

In order to balance the first two terms on the right side of
\eqref{cccc} we let
\begin{equation*}
H=c_5(K^{\frac{2(5q+4)d^3+(3q+19)d^2-(13q+24)d-6}{2(d-1)}}T^{-d/2}M_*^{(q/2+1)d})^{Q/(Q+(Q-1)d)}.
\end{equation*}
We then need to check that \eqref{H} is satisfied with this choice
of $H$. First, $H>1$ since we can assume
\begin{equation}
T<c_7 K^{(2(5q+4)d^3+(3q+19)d^2-(13q+24)d-6)/(d(d-1))} M_*^{q+2}
\label{remove1}
\end{equation}
with a sufficiently small $c_7$ (otherwise the trivial bound of
$S(T, M_*; G, F)$, {\it i.e.} $M_*^d$, is better than \eqref{kkk}).
Second, the assumption \eqref{restriction} implies the second
inequality in \eqref{H}.

With the choice of $H$ as above and \eqref{remove1}, we get
\begin{equation}
\begin{split}
H^{-2+2/Q}T^{-d/2}&M_*^{d(Q+q/2)}\leqslant\\
&K^{-\frac{2(5q+4)d^3+(3q+19)d^2-(13q+24)d-6}{2(d-1)}}M_*^{(Q-1)d}H^{d-1}.
\end{split}\label{remove3}
\end{equation}
It then follows from \eqref{H} and \eqref{remove3} that
\begin{equation*}
\textrm{E}\leqslant M_*^{Qd}H^{-1}.
\end{equation*}
Applying this bound to \eqref{cccc} finally yields the desired bound
\eqref{kkk}.
\end{proof}


\subsection*{Acknowledgments}
I would like to express my sincere gratitude to Professor Andreas
Seeger for his valuable advice and great help. I also thank my
friends Weiyong He and Yuanqi Wang for helpful discussion on some
geometric questions.


\end{document}